\documentclass{kms-j}

\newcommand{\publname}{}
\issueinfo{}
  {}
  {}
  {}
\pagespan{1}{}
\copyrightinfo{}
  {}

\usepackage{graphicx}
\allowdisplaybreaks

\theoremstyle{plain}
\newtheorem{theorem}{Theorem}[section]

\newtheorem{lemma}[theorem]{Lemma}

\theoremstyle{definition}
\newtheorem{definition}{Definition}

\theoremstyle{remark}
\newtheorem{remark}[theorem]{Remark}

\begin{document}

\title[Stochastic semilinear evolution  equations]
{Note on abstract stochastic semilinear evolution  equations}

\author[T.V. T\d{a}]{T$\hat{\rm \text{o}}$n Vi$\hat{\d{e}}$t  T\d{a}}

\address{T$\hat{\rm \text{o}}$n Vi$\hat{\d{e}}$t T\d{a} \\
Promotive Center for International Education and Research of Agriculture \\
Faculty of Agriculture\\
 Kyushu University \\
 Higashi-ku, Fukuoka 812-8581, Japan\\
  }
\email{tavietton[at]agr.kyushu-u.ac.jp}

\thanks{This work was supported by JSPS KAKENHI Grant Number 20140047.}


\begin{abstract}
This paper is devoted to studying abstract stochastic semilinear  evolution  equations with additive noise in Hilbert spaces.  First, we prove   the existence of unique local mild solutions and show their regularity. Second, we show the regular dependence of the solutions on initial data. Finally,  some applications to stochastic partial differential equations are presented.
\end{abstract}

\maketitle

\section {Introduction}
Many interesting phenomena in the real world can be described by a system of nonlinear parabolic evolution equations. These equations generally not only generate a dynamical system but also  a global attractor even a finite-dimensional attractor. Such an attractor then suggests that the phenomena  enjoy some robustness in a certain abstract sense. Some may be the pattern formation and others may be the specific structure creation (\cite{Nguyen1,Nguyen2,Nguyen3}). In these cases, one of main issues is to study the robustness of the final states of system. It is therefore quite natural in order to investigate the robustness to consider an advanced version of stochastic parabolic evolution equations.

In this paper, 
we  study the Cauchy problem for an abstract stochastic semilinear evolution  equation: 
\begin{equation} \label{P1}
\begin{cases}
dX+AXdt=[F_1(X)+F_2(t)]dt+ G(t)dW(t), \hspace{1cm} t\in(0,T],\\
X(0)=\xi
\end{cases}
\end{equation}
in a  separable Hilbert space $H$. Here, $A\colon\mathcal D(A)\subset H\to H$ is a sectorial  operator. The process  $W$ is a cylindrical Wiener process on a  separable Hilbert space $U$, and is defined on a filtered  probability space  $(\Omega, \mathcal F,\mathcal F_t,\mathbb P)$. The function
 $F_1$ is  measurable  from $(\Omega\times H, \mathcal F_T\times \mathcal B(H))$  into $(H,\mathcal B(H))$.  Meanwhile, $F_2$ and $G$ are   measurable functions  from $([0,T], \mathcal B([0,T]))$ into $(H,\mathcal B(H))$  and 
 $(L_2(U;H),\mathcal B(L_2(U;H)))$, respectively (here  
   $L_2(U;H)$ denotes the space of all Hilbert-Schmidt operators from $U$ to $H$). The initial value 
  $ \xi$ is an $H$-valued $\mathcal  F_0$-measurable  random variable.

This kind of evolution equations has been investigated by several authors (see \cite{Brzezniak,Brzezniak1,pratoFlandoliPriolaRockner,prato,Ichikawa,Ton1,Ton1.5,Ton4,van2,Neerven}, and references therein). 
Under the Lipschitz continuity  and  linear growth conditions on $F_1$, Ichikawa \cite{Ichikawa} and Da Prato-Zabczyk \cite{prato} proved the existence  of unique global mild solutions in $L_p([0,T]; H)$, Neerven-Veraar-Weis \cite{Neerven} showed the existence of unique strong solutions in $L_p([0,T]; \mathcal D(A)).$
 In Banach space setting, Brzezniak \cite{Brzezniak1} (see also T\d{a}-Yagi \cite{Ton4}) showed the existence of maximal local mild solutions. The space-time regularity of solutions to 
 \eqref{P1} has however not been developed well for the case where the domain of $F_1$ is a subset of $H$, and $F_1$does not satisfy the linear growth condition. Such a case occurs very often in many phenomena described by partial differential equations (PDEs) (see e.g., Yagi \cite{yagi}).

In the present paper, we study the equation  \eqref{P1}, where $F_1$ is defined on a subset of $H$  and satisfies a Lipschitz condition (see {\rm (H3)} in Section \ref{section5}). We  prove  the existence and uniqueness  of local mild solutions. We also  show the space-time regularity and dependence on initial data of the solutions. 
Here, the local  solutions are constructed on nonrandom intervals. Note that previous results in  \cite{Brzezniak1,Ton4} show that local solutions are defined on random intervals.

For the study, we  use the semigroup approach. Let us explain this approach. Consider the Cauchy problem for a linear evolution equation
\begin{equation*}
\begin{cases}
\frac{dX}{dt}+AX=F(t), \quad\quad 0<t\leq T,\\
X(0)=X_0.
\end{cases}
\end{equation*}
Hille \cite{Hille} and Yosida \cite{Yosida} invented the semigroup $S(t)=e^{-tA}$ generated by a linear operator $(-A),$ which directly  provides a fundamental solution to the Cauchy problem 
$$X(t)=S(t)X_0+\int_0^t S(t-s)F(s)ds.$$
Similarly, a solution to the Cauchy problem for a nonlinear evolution equation
\begin{equation*}
\begin{cases}
\frac{dX}{dt}+AX=F(t,X), \quad\quad 0<t\leq T,\\
X(0)=X_0,
\end{cases}
\end{equation*}
can be obtained as a solution of an integral equation 
$$X(t)=S(t)X_0+\int_0^t S(t-s)F(s,X(s))ds.$$
By these  formulas, one can get  important information on solutions such as uniqueness, regularity, smoothing effect and so forth. Especially, for nonlinear problems one can derive Lipschitz continuity of solutions with respect to the initial values, even their Fr\'{e}chet differentiability.

The organization of this paper is as follows. Section \ref{section2} is  preliminary. We review some notions such as weighted H\"{o}lder continuous function spaces, analytical semigroups generated by sectorial operators, and  cylindrical Wiener processes.
 Section \ref{section5}  presents our main results. We assume  that the function $F_1$ is defined only on a subset of the space $H$,  $\mathcal D(F_1)= \mathcal D(A^\eta)$ for some $0<\eta<1$, and  satisfies a Lipschitz condition  (see {\rm (H3)}). We suppose further  that  the initial value  $\xi$ takes values in a smaller space, namely,  $\mathcal D(A^\beta)$.
Theorem \ref{Thm5} gives the existence and uniqueness of local solutions as well as its temporal and spacial regularity. Theorem  \ref{Thm7}  gives  the  regularity of the expectation of  local solutions. Theorem \ref{Thm6}   shows  the regular  dependence of solutions on initial data.  Finally, Section \ref{section4} gives some applications to stochastic PDEs.

\section{Preliminary}  \label{section2}
\subsection{Weighted H\"{o}lder continuous function spaces}
Let us review the notion of weighted H\"{o}lder continuous function spaces  $\mathcal F^{\beta, \sigma}((0,T]; H)$ 
 for two exponents $0<\sigma<\beta< 1.$ This kind of spaces is introduced  by Yagi \cite{yagi}.

 The space  $\mathcal F^{\beta, \sigma}((0,T]; H)$ consists of  $H$-valued continuous functions $F$ on $(0,T]$   with the following three properties:
\begin{itemize}
  \item [\rm (i)] 
   \begin{equation}  \label{P2}
  t^{1-\beta} F(t)  \text{  has a limit as   } t\to 0.
  \end{equation}
    \item [\rm (ii)] $F$ is H\"{o}lder continuous with  exponent $\sigma$ and   weight $s^{1-\beta+\sigma}$, i.e.
  \begin{equation} \label{P3}
\begin{aligned}
\sup_{0\leq s<t\leq T} & \frac{s^{1-\beta+\sigma}\|F(t)-F(s)\|}{(t-s)^\sigma}\\
&=\sup_{0\leq t\leq T}\sup_{0\leq s<t}\frac{s^{1-\beta+\sigma}\|F(t)-F(s)\|}{(t-s)^\sigma}<\infty.
\end{aligned}
\end{equation}
  \item [\rm (iii)] 
  \begin{equation} \label{P4}
   \lim_{t\to 0} w_{F}(t)=0, 
  \end{equation}
  where $w_{F}(t)=\sup_{0\leq s  <t}\frac{s^{1-\beta+\sigma}\|F(t)-F(s)\|}{(t-s)^\sigma}$.
\end{itemize}  
It is easily seen that   $\mathcal F^{\beta, \sigma}((0,T];E)$ is a Banach space with   norm
$$\|F\|_{\mathcal F^{\beta, \sigma}}=\sup_{0\leq t\leq T} t^{1-\beta} \|F(t)\|+ \sup_{0\leq s<t\leq T} \frac{s^{1-\beta+\sigma}\|F(t)-F(s)\|}{(t-s)^\sigma}.$$

 Clearly,  for  $F\in \mathcal F^{\beta, \sigma}((0,T];H),$
\begin{equation} \label{P5}  
\begin{cases}
\|F(t)\|\leq \|F\|_{\mathcal F^{\beta, \sigma}} t^{\beta-1}, \hspace{2cm} 0<t\leq T,\\
\begin{aligned}
 \|F(t)-F(s)\| & \leq w_{F}(t) (t-s)^{\sigma} s^{\beta-\sigma-1}\\
& \leq \|F\|_{\mathcal F^{\beta, \sigma}} (t-s)^{\sigma} s^{\beta-\sigma-1},  \hspace{1cm} 0<s\leq t\leq T.
\end{aligned}
\end{cases}
\end{equation}
\begin{remark}  \label{remark1}
\begin{itemize}
  \item [{\rm (a)}]
The space $\mathcal F^{\beta, \sigma}((0,T];H)$ is not  a trivial space.  The function $F$ defined by  $F(t) =t^{\beta-1} f(t), 0<t\leq T,$ belongs to this space, where  $f$ is any $E$-valued  function  such that
 $f\in \mathcal C^\sigma([0,T];E)$
  and $f(0)=0.$
\item [{\rm (b)}]  The space $\mathcal F^{\beta, \sigma}((a,b];E)$, $0\leq a<b<\infty$, is defined in a similar way. For more details, see \cite{yagi}.
\end{itemize}
\end{remark}

\subsection{Sectorial operators and analytical semigroups}
A densely defined, closed linear operator $A$ is said to be sectorial if it satisfies two conditions:
\begin{itemize}
  \item [(\rm{H1})] The spectrum $\sigma(A)$ of $A$ is  contained in an open sectorial domain $\Sigma_{\varpi}$: 
\begin{equation*} \label{H1} 
\sigma(A) \subset  \Sigma_{\varpi}=\{\lambda \in \mathbb C: |\arg \lambda|<\varpi\}, \quad \quad 0<\varpi<\frac{\pi}{2}.
       \end{equation*}
  \item [(\rm{H2})] The resolvent of $A$ satisfies the estimate  
\begin{equation*} \label{H2}
          \|(\lambda-A)^{-1}\| \leq \frac{M_{\varpi}}{|\lambda|}, \quad\quad\quad \quad   \lambda \notin \Sigma_{\varpi}
     \end{equation*}
     with some  constant $M_{\varpi}>0$ depending only on the angle $\varpi$.
     \end{itemize}

Let $A$ be a  sectorial operator. The fractional powers $A^\theta, -\infty<\theta<\infty,$ are then defined as follows. For each complex number $z$ such that ${\rm Re}\, z>0$, $A^{-z}$ is defined by using the Dunford integral in $\mathcal L(H)$:
$$A^{-z}=\frac{1}{2\pi i} \int_\gamma \lambda^{-z} (\lambda-A)^{-1} d\lambda.$$
Here, $\gamma=\gamma_-\cup \gamma_0 \cup \gamma_+$ is an integral contour surrounding the spectrum $\sigma(A)$ counterclockwise in the domain $\mathbb C\setminus (-\infty,0] \cap \mathbb C\setminus\sigma(A)$ of the complex plane:
\begin{align*}
\gamma_\pm \colon \lambda=\rho e^{\pm i \varpi}, \hspace{2cm} \frac{\|A^{-1}\|^{-1}}{2} \leq \rho<\infty,
\end{align*}
and 
\begin{align*}
\gamma_0 \colon \lambda=\frac{\|A^{-1}\|^{-1}}{2} e^{ i \varphi}, \hspace{1cm}  -\varpi \leq \varphi\leq \varpi.
\end{align*}
It is known that $A^{-z}$ is one to one for ${\rm Re}\, z>0$. The following definition is thus meaningful:
$$A^z=(A^{-z})^{-1} \hspace{1cm} \text{ for } {\rm Re}\, z>0.$$
In addition, it is natural to define $A^0=1$.  
In this way, for every real number $-\infty<\theta<\infty,$ $A^\theta$ has been defined. For more detail on fractional powers, see \cite{yagi}.

The following lemma shows useful estimates for fractional powers and the semigroup generated by a sectorial operator.
\begin{lemma}
Let {\rm (H1)} and {\rm (H2)}  be satisfied. Then,
\begin{itemize}
\item  [\rm (i)] $(-A)$ generates an analytical  semigroup $S(t)=e^{-tA}.$
\item  [\rm (ii)] For $0\leq \theta  <\infty,$
\begin{equation} \label{P6}
\|A^\theta S(t)\| \leq \iota_\theta t^{-\theta}, \hspace{2cm} 0<t<\infty,   
\end{equation}
 where $ \iota_\theta=\sup_{0\leq  t<\infty} t^\theta \|A^\theta S(t)\|<\infty$. In particular, there exists $\nu>0$ such that
\begin{equation} \label{P7}
\|S(t)\| \leq \iota_0 e^{-\nu t}\leq \iota_0, \hspace{2cm}  0\leq t<\infty.
\end{equation}
 \item  [\rm (iii)]  For  $0<\theta\leq 1,$ 
\begin{equation} \label{P8}
\|[S(t)-I]A^{-\theta}\|\leq \frac{\iota_{1-\theta}}{\theta} t^\theta, \hspace{2cm}  0\leq t<\infty.
\end{equation}
\end{itemize}
\end{lemma}
For the proof, see \cite{yagi}.

To end this subsection, let us recall a result  presented in  \cite{yagi0,yagi}.
\begin{theorem}\label{Thm2}
Let {\rm (H1)} and {\rm (H2)} be satisfied.  
Let 
$$x\in \mathcal D(A^\beta) \, \text{ and } \,F\in \mathcal F^{\beta,\sigma}((0,T];H)$$
for some  $ 0<\sigma<\beta < 1.$ 
Set
$$I(t)=S(t) x+\int_0^t S(t-s)F(s)ds,\hspace{1cm} 0\leq t\leq T,$$
where $S(t)$ is the analytical semigroup generated by $(-A)$.
Then, $I$ possesses the  properties:
$$A^\beta I \in \mathcal C([0,T];H), $$ 
$$\frac{ dI}{dt}, AI \in \mathcal F^{\beta,\sigma}((0,T];H)$$
with the estimate
$$\Big|\Big|\frac{dI}{dt}\Big|\Big|_{\mathcal F^{\beta,\sigma}}+\|A^\beta I\|_{\mathcal C}+\|AI\|_{\mathcal F^{\beta,\sigma}}\leq C[\|A^\beta x\|+\|F\|_{\mathcal F^{\beta,\sigma}}],$$
where $C>0$ is some constant depending only on $\beta$ and $\sigma$.
\end{theorem}
\subsection{Cylindrical Wiener process}
Let us review a central notion to the theory of stochastic evolution equations, namely,  cylindrical Wiener processes on the Hilbert space $U$. First, we recall the definition of  $Q$\,-\,Wiener processes on Hilbert spaces (see  \cite{CurtainFalb}).

\begin{definition} \label{definition1}
An $U$-valued stochastic process $W$ defined on a filtered  probability space  $(\Omega, \mathcal F,\mathcal F_t,\mathbb P)$ is a $Q$\,-\,Wiener process if 
\begin{itemize}
  \item $W(0)=0$ \hspace{1cm} a.s.
  \item $W$ has continuous sample paths
  \item $W$ has independent increments
  \item The law of $W(t)-W(s), 0<s\leq t,$ is a Gaussian measure on $U$ with mean $0$ and covariance $(t-s)Q,$ where $Q$ is a symmetric nonnegative nuclear operator in $\mathcal L(U)$
\end{itemize}
\end{definition}

\begin{remark}
\begin{itemize}
  \item [(i)] When $U$ is the real line $\mathbb R$, the operator $Q$ is just a positive number $q$. The  $Q$\,-\,Wiener process is then a  Brownian motion on $\mathbb R$. When $U=\mathbb R^n (n=2,3,4,\dots)$, $Q$ is an $n\times n$ positive definite matrix. In this case, the $Q$\,-\,Wiener process is a Brownian motion in $\mathbb R^n$.
  \item [(ii)] The operator  $Q$ is not only a bounded linear operator but also a nuclear operator, i.e. its trace is finite:
$${\rm Tr} \,(Q)=\sum_{i=1}^\infty \langle Qe_i,e_i \rangle <\infty,$$
here $\{e_i\}_{i=1}^\infty$ is a complete orthonormal basis in $U$.
 \end{itemize}
\end{remark}

 Let us now fix a larger Hilbert space $U_1$ such that $U$ is embedded continuously into $U_1$ and the embedding 
$J\colon U\to U_1$ is Hilbert-Schmidt (i.e. $\sum_{i=1}^\infty \|Je_i\|_{U_1}^2<\infty).$ For example (see  \cite{Hairer}), we take $U_1$ to be the closure of $U$ under the norm
$$\|h\|_{U_1}=\left[\sum_{n=1}^\infty \frac{\langle h,e_n\rangle_U^2}{n^2}\right]^{\frac{1}{2}}.$$
For every $u\in U_1$, we have
\begin{align*}
\langle JJ^*e_m, u\rangle_{U_1}
&=\langle J^*e_m, J^*u\rangle_U\\
&=\langle\sum_{k=1}^\infty e_k \langle J^*e_m,e_k\rangle_U, \sum_{k=1}^\infty e_k \langle J^*u,e_k\rangle_U\rangle_U\\
&= \sum_{k=1}^\infty \langle J^*e_m,e_k\rangle_U  \langle J^*u,e_k\rangle_U
=\sum_{k=1}^\infty \langle e_m,Je_k\rangle_{U_1}  \langle u,Je_k\rangle_{U_1}\\
&=\sum_{k=1}^\infty \langle e_m,e_k\rangle_{U_1}  \langle u,e_k\rangle_{U_1}
=\|e_m\|_{U_1}^2  \langle u,e_m\rangle_{U_1}\\
&= \frac{1}{m^2} \langle u,e_m\rangle_{U_1}, \hspace{2cm}    m=1,2,3,\dots
\end{align*}
Therefore, $JJ^*e_m=\frac{1}{m^2} e_m $ for  $m=1,2,3,\dots$ As a consequence,
$${\rm Tr}\, (JJ^*)=\sum_{m=1}^\infty \langle JJ^*e_m,e_m\rangle=\sum_{m=1}^\infty \frac{1}{m^2}<\infty.$$
Thus, $JJ^*$ is a nuclear operator.

Based on  the operator $JJ^*$ , one can define a cylindrical Wiener process.  The following definition is taken from  \cite{prato,Hairer}.
\begin{definition} 
The $U_1$-valued Wiener process in Definition \ref{definition1} with covariance $Q=JJ^*$  is called a cylindrical Wiener process on $U$.
\end{definition}

The $H$-valued stochastic integrals against a cylindrical Wiener process on $U$ is then constructed in the same way as what is usually done in finite dimensions. In  \cite{prato},  the stochastic integrals $\int_0^T \Phi(s)dW(s)$ are constructed for  integrand   $\Phi$ in   $\mathcal N^2(0,T;L_2(U;H))$, the space of all $L_2(U;H)$-valued predictable stochastic processes $\Phi$ on $[0,T]$ such that 
$$\mathbb E \int_0^T \|\Phi(s)\|_{L_2(U;H)}^2 ds<\infty.$$
It is known that the class  $\mathcal N^2(0,T;L_2(U;H))$  is independent of the space $U_1$ chosen. Furthermore, stochastic integrals can be extended to $L_2(U;H)$-valued predictable stochastic processes $\Phi$ satisfying a weaker condition:
$$\mathbb P\left\{ \int_0^T \|\Phi(s)\|_{L_2(U;H)}^2 ds<\infty\right\}=1.$$
The set of all such processes is denoted by $\mathcal N(0,T;L_2(U;H)).$ The readers can find  properties of stochastic integrals against cylindrical Wiener processes in \cite{prato}. Those are similar to ones of the usual stochastic integrals. 

The following known result is used very often.
\begin{theorem}
Let $G\in \mathcal N^2(0,T;L_2(U;H))$.  
Let $B$ be a closed linear operator on $H$ such that
 $$          \mathbb E\int_0^T ||BG(s)||_{L_2(U;H)}^2 ds<\infty. $$ 
          Then,
$$B\int_0^T G(t)dW(t)=\int_0^T BG(t)dW(t) \hspace{1cm} a.s.$$
\end{theorem}
For the proof, see  \cite{prato}.

\subsection{Mild solutions}
Let us introduce  a definition of  mild solutions to \eqref{P1} (see \cite{prato,Ichikawa}).
\begin{definition}\label{Def1}
Let $F_2$ and $G$ be $H$-valued and $L_2(U;H)$-valued functions satisfying the conditions:
$$\int_0^t \|S(t-s)F_2(s)\| ds<\infty, \hspace{2cm} 0\leq t\leq T, $$ 
and 
$$   \int_0^t \|S(t-s)G(s)\|_{L_2(U;H)}^2 ds<\infty, \hspace{2cm} 0\leq t\leq T.  $$
A predictable $H$-valued process $X$ on $ [0,T]$ is called a  mild solution  of \eqref{P1} if 
$$  
\int_0^T \|S(t-s)F_1(X(s))\| ds<\infty   \hspace{2cm} \text{ a.s., } 
$$  
and 
\begin{align*}
 \label{DefinitionGlobalMildSolutions}X(t)=&S(t)\xi +\int_0^tS(t-s) [F_1(X(s))+F_2(s)]ds\\
&+ \int_0^t S(t-s) G(s)dW(s) \hspace{3cm} \text{a.s., } 0< t\leq T.\notag
\end{align*}
\end{definition}

In order to study the H\"{o}lder continuity of solutions,  the Kolmogorov test is useful.
\begin{theorem}  \label{Thm1}
Let $\zeta$ be  an $H$-valued stochastic process on $[0,T].$ Assume that  for some  $c>0$ and $ \epsilon_i>0 \, (i=1,2),$ 
\begin{equation*}
\mathbb E\|\zeta(t)-\zeta(s)\|^{\epsilon_1}\leq c |t-s|^{1+\epsilon_2}, \hspace{1cm}  0\leq s,t\leq T.
\end{equation*}
Then, $\zeta$ has  a version whose $\mathbb P$-almost all trajectories are H\"{o}lder continuous functions with an  arbitrarily smaller  exponent than $\frac{\epsilon_2}{\epsilon_1}$.
\end{theorem}
For the proof, see e.g.,  \cite{prato}.

\section{Main results} \label{section5}
In this section, we  prove the  existence and  uniqueness  of local mild solutions to  \eqref{P1} and show their regularity (Subsection \ref{sub3.1}). We then  show regular dependence of solutions on initial data (Subsection \ref{sub3.2}).

Let fix constants $\eta, \beta, \sigma $ such that
\begin{equation*}
\begin{cases}
0<\eta<\frac{1}{2}, \\
 \max\{0, 2\eta-\frac{1}{2}\}<\beta<\eta, \\
 0<\sigma <\beta.
\end{cases}
\end{equation*}
Assume  that
\begin{itemize}
\item [{\rm (H3)}] $F_1\colon \mathcal D(A^\eta)\subset H \to H$  satisfies a Lipschitz condition of the form
     \begin{equation*} 
        \|F_1(x)-F_1(y)\|\leq c_{F_1}  \|A^\eta(x-y)\| \hspace{1cm}  \text{ a.s., }  x,y\in \mathcal D(A^\eta), 
      \end{equation*}
where $c_{F_1}$ is some positive constant.
\item [{\rm (H4)}] $F_2 \in \mathcal F^{\beta,\sigma} ((0,T];H). $
\item [{\rm (H5)}] $ G\in B([0,T];L_2(U;H)).$
\end{itemize}
Here, $B([0,T];L_2(U;H))$ is the space of uniformly bounded $L_2(U;H)$-valued functions on $[0,T]$  with the supremum norm:
$$\|G\|_{B([0,T];L_2(U;H))}=\sup_{0\leq t\leq T} \|G(t)\|_{L_2(U;H)}.$$

\begin{lemma}\label{Thm4}
Let {\rm (H1)}, {\rm (H2)}, {\rm (H3)}, {\rm (H4)} and {\rm (H5)}   be satisfied. Then,
\begin{itemize}
\item [{\rm (i)}] For  $0\leq \theta<1,$ 
$$\int_0^t\|A^\theta S(t-s)F_2(s)\|ds\leq \iota_\theta \|F_2\|_{\mathcal F^{\beta,\sigma}}  B(\beta,1-\theta) t^{\beta-\theta}, \hspace{0.5cm} 0<t \leq T,$$
where $B(\cdot,\cdot)$ denotes the beta function. 
As a consequence, 
$$A^\theta \int_0^t  S(t-s)F_2(s)ds=\int_0^t A^\theta S(t-s)F_2(s)ds$$
and $A^\theta \int_0^\cdot  S(\cdot-s)F_2(s)ds$ is continuous on $(0,T]$. Furthermore, if $\theta<\beta,$ then $A^\theta \int_0^\cdot  S(\cdot-s)F_2(s)ds$ is also continuous at  $t=0$.
  \item [{\rm (ii)}] For  $0 \leq \theta<\frac{1}{2},$ 
$$\int_0^t\|A^\theta S(t-s)G(s)\|_{L_2(U;H)}^2 ds\leq \frac{ \iota_\theta^2 t^{1-2\theta}\|G\|_{B([0,T]; L_2(U;H))}^2}{1-2\theta}, \quad 0\leq t\leq T.$$ 
As a consequence, the stochastic convolution $W_G$  defined by 
 $$W_G(t)=\int_0^t  S(t-s)G(s)dW(s), \hspace{1cm} 0\leq t\leq T$$ 
satisfies
$$A^\theta W_G(t)=\int_0^t A^\theta S(t-s)G(s)dW(s)   \hspace{1cm} \text{a.s.}, 0\leq t\leq T.$$
Furthermore,  $A^\theta W_G$  is continuous on $[0,T]$.
\item [{\rm (iii)}]  For any $ 0< \gamma <\frac{1+2\beta}{4}-\eta,$   $W_G$ has the regularity:
$$A^\eta W_G\in \mathcal C^\gamma ((0,T];H) \hspace{1cm} \text{a.s.}$$
 
\end{itemize}
\end{lemma}

\begin{proof}
First, let us prove {\rm (i)}. It follows from  \eqref{P5} and \eqref{P6} that 
\begin{align*}
&\int_0^t\|A^\theta S(t-s)F_2(s)\|ds\\
&\leq \int_0^t\|A^\theta S(t-s)\| \|F_2(s)\|ds \notag\\
&\leq \|F_2\|_{\mathcal F^{\beta,\sigma}} \iota_\theta\int_0^t(t-s)^{-\theta} s^{\beta-1}ds\notag\\
&= \|F_2\|_{\mathcal F^{\beta,\sigma}} \iota_\theta t^{\beta-\theta}\int_0^1u^{\beta-1}(1-u)^{-\theta}du\notag\\
&=\iota_\theta \|F_2\|_{\mathcal F^{\beta,\sigma}}  B(\beta,1-\theta) t^{\beta-\theta}, \hspace{2cm} 0<t\leq T.  
\end{align*}
 Hence,   $\int_0^\cdot A^\theta S(\cdot -s)F_2(s)ds$ is continuous on $(0,T].$ It  is  also continuous at $t=0$ if $\theta<\beta$. Since $A^\theta$ is closed, the statements in  {\rm (i)} follow.
 
Let us next verify  {\rm (ii)}. Thanks to \eqref{P6}, 
\begin{align*}
&\int_0^t\|A^\theta S(t-s)G(s)\|_{L_2(U;H)}^2 ds \\
&\leq \int_0^t\|A^\theta S(t-s)\|^2 \|G(s)\|_{L_2(U;H)}^2 ds\\
&\leq  \iota_\theta^2 \int_0^t (t-s)^{-2\theta} \|G\|_{B((0,T]; L_2(U;H))}^2ds\\
&=\frac{ \iota_\theta^2 t^{1-2\theta}\|G\|_{B([0,T]; L_2(U;H))}^2}{1-2\theta}  <\infty, \hspace{1cm}0\leq t\leq T.
\end{align*}
The process  $\int_0^\cdot  A^\theta S(\cdot -s)G(s)dW(s)$ is therefore well-defined. The definition of stochastic integrals then provides that this process is a continuous martingale on $[0,T]$. Since $A^\beta$ is closed, {\rm (ii)}  follows.

The proof for {\rm (iii)} is  similar to one in \cite{prato,Ton1}. So, we omit it.
\end{proof}

\subsection{Existence and regularity of  solutions} \label{sub3.1}
Let us first prove the existence of unique local mild solutions to \eqref{P1}  and show their space-time regularity.
\begin{theorem}\label{Thm5}
Let {\rm (H1)}, {\rm (H2)}, {\rm (H3)}, {\rm (H4)} and {\rm (H5)}   be satisfied.
Let $\xi\in \mathcal D(A^\beta)$ such that $\mathbb E\|A^\beta\xi\|^2<\infty$.  Then, \eqref{P1} possesses a unique local mild solution $X$ in the function space:
\begin{equation}\label{P10}
X\in  \mathcal C([0,T_{loc}];\mathcal D(A^\beta)), \quad A^\eta X\in \mathcal C^\gamma ((0,T_{loc}];H)  \hspace{1cm} \text{a.s.}
\end{equation}
for any $ 0<\gamma <\frac{1+2\beta}{4}-\eta.$ 
 Furthermore, $X$ satisfies the estimate:
\begin{equation}\label{P11}
\mathbb E \|A^\beta X(t)\|^2 +t^{2(\eta-\beta)} \mathbb E \|A^\eta X(t)\|^2 \leq C_{F_1,F_2,G,\xi},    \hspace{1cm} 0\leq t\leq T_{loc}.
\end{equation}
Here, $T_{loc}$ and  $C_{F_1,F_2,G,\xi}$ are non-random constants depending  on the exponents and  $\mathbb E \|F_1(0)\|^2,$ $ \mathbb E \|A^\beta\xi\|^2,$  $\|F_2\|_{\mathcal F^{\beta,\sigma}}^2$, $ \|G\|_{B([0,T];L_2(U;H))}^2.$ 
\end{theorem}

\begin{proof}
We  use the fixed point theorem for contractions to prove the existence and uniqueness of  local solutions.
For each $0<S\leq T$,  set the underlying space:
\begin{align*}
\Xi (S)=&\{Y\in \mathcal C((0,S];\mathcal D(A^\eta)) \cap  \mathcal C([0,S];\mathcal D(A^\beta)) \text{  such that   }\\
& \sup_{0<t\leq S} t^{2(\eta-\beta)} \mathbb E\|A^\eta Y(t)\|^2+ \sup_{0\leq t\leq S}\mathbb E\|A^\beta Y(t)\|^2 <\infty \}.
\end{align*}
Up to indistinguishability, $\Xi (S)$  is then a Banach space with norm
\begin{equation} \label{P13}
\|Y\|_{\Xi (S)}=\Big[\sup_{0<t\leq S} t^{2(\eta-\beta)} \mathbb E\|A^\eta Y(t)\|^2+ \sup_{0\leq t\leq S}\mathbb E\|A^\beta Y(t)\|^2\Big]^{\frac{1}{2}}. 
\end{equation}

Let fix a constant $\kappa>0$  such that 
\begin{equation}  \label{P14}
\frac{\kappa^2}{2}> C_1\vee C_2,
\end{equation}
 where two constants $C_1$ and $C_2$ will be fixed below.
Consider a subset $\Upsilon(S) $ of $\Xi (S)$ which consists of  functions $Y\in \Xi (S)$  such that
\begin{equation}  \label{P15}
\max\left\{\sup_{0<t\leq S} t^{2(\eta-\beta)} \mathbb E\|A^\eta Y(t)\|^2,
 \sup_{0\leq t\leq S}\mathbb E\|A^\beta Y(t)\|^2\right\} \leq \kappa^2.
\end{equation}
Obviously, $\Upsilon(S) $  is a nonempty closed subset of $\Xi (S)$.

For $Y\in \Upsilon(S)$, we define a function on $[0,S]:$
\begin{align}  
\Phi Y(t)=&S(t) \xi+\int_0^t S(t-s)[F_1(Y(s))+F_2(s)]ds \label{P16}\\
&+ \int_0^t S(t-s)G(s) dW(s). \notag
\end{align}
Our goal is then to verify that $\Phi$ is a contraction mapping from $\Upsilon(S)$ into itself, provided that $S$ is sufficiently small, and that the fixed point of $\Phi$ is the desired solution of \eqref{P1}. For this purpose, we divide the proof into four steps.

{\bf Step 1}. Let us show that $ \Phi Y \in \Upsilon(S)$ for  $Y \in \Upsilon(S). $

 Let  $Y\in \Upsilon(S)$. 
Due to {\rm (H3)} and \eqref{P15}, we observe that
\begin{align}
&\mathbb E\|F_1(Y(t))\|^2 \notag \\
&\leq \mathbb E[c_{F_1}\|A^\eta Y(t)\|+\|F_1(0)\|]^2\notag \\
&\leq 2[c_{F_1}^2 \mathbb E \|A^\eta Y(t)\|^2 +\mathbb E\|F_1(0)\|^2]\notag\\
&\leq 2[c_{F_1}^2 \kappa^2 t^{2(\beta-\eta)}   +\mathbb E\|F_1(0)\|^2], \hspace{2cm} 0<t\leq S.  \label{P17}
\end{align}

First, we  verify that  $ \Phi Y$ satisfies \eqref{P15}.  For $\beta\leq \theta < \frac{1}{2}$,  \eqref{P16} gives
\begin{align*}
&t^{2(\theta-\beta)}\mathbb E\|A^\theta\{\Phi Y\}(t)\|^2   \notag\\
 \leq & 3t^{2(\theta-\beta)}\mathbb E \Big[ \|A^\theta S(t) \xi\|^2+\Big|\Big|\int_0^tA^\theta S(t-s)[F_1(Y(s))+F_2(s)]ds\Big|\Big|^2 \notag\\
&+\Big|\Big|\int_0^tA^\theta S(t-s) G(s) dW(s)\Big|\Big|^2 \Big]  \notag\\
\leq & 3t^{2(\theta-\beta)}  \|A^{\theta-\beta} S(t)\|^2\mathbb E\|A^\beta \xi\|^2+6t^{2(\theta-\beta)}\mathbb E\Big|\Big|\int_0^tA^\theta S(t-s)F_1(Y(s))ds\Big|\Big|^2 \notag\\
&+6t^{2(\theta-\beta)}\mathbb E\Big|\Big|\int_0^t A^\theta S(t-s)F_2(s)ds\Big|\Big|^2 \notag\\
& +3t^{2(\theta-\beta)} \int_0^t\|A^\theta S(t-s) G(s)\|_{L_2(U;H)}^2 ds. \notag
\end{align*}
On the account of \eqref{P5}, \eqref{P6} and Lemma \ref{Thm4}, we have
\begin{align*}
&t^{2(\theta-\beta)}\mathbb E\|A^\theta\{\Phi Y\}(t)\|^2   \notag\\
\leq  &  3\iota_{\theta-\beta}^2 \mathbb E\|A^\beta \xi\|^2+6t^{2(\theta-\beta)} \iota_\theta^2 \mathbb E \Big[\int_0^t (t-s)^{-\theta}\|F_1(Y(s))\|ds\Big]^2\notag\\
&+6 \iota_\theta^2 \|F_2\|_{\mathcal F^{\beta,\sigma}}^2 B(\beta,1-\theta)^2 +\frac{3 \iota_\theta^2  \|G\|_{B([0,T];L_2(U;H))}^2   t^{1-2\beta}}{1-2\theta}\notag\\
\leq  &  3\iota_{\theta-\beta}^2 \mathbb E\|A^\beta \xi\|^2+6t^{1+2(\theta-\beta)} \iota_\theta^2  \int_0^t (t-s)^{-2\theta}\mathbb E \|F_1(Y(s))\|^2ds\notag\\
&+6 \iota_\theta^2 \|F_2\|_{\mathcal F^{\beta,\sigma}}^2 B(\beta,1-\theta)^2 +\frac{3 \iota_\theta^2  \|G\|_{B([0,T];L_2(U;H))}^2   t^{1-2\beta}}{1-2\theta}.  \notag
\end{align*}

 The second term in the right-hand side of the latter inequality is estimated by using \eqref{P17}:
\begin{align*}
&6t^{1+2(\theta-\beta)} \iota_\theta^2  \int_0^t (t-s)^{-2\theta}\mathbb E \|F_1(Y(s))\|^2ds   \notag\\
\leq &  12 t^{1+2(\theta-\beta)} \iota_\theta^2  \int_0^t (t-s)^{-2\theta}[c_{F_1}^2 \kappa^2 s^{2(\beta-\eta)}   +\mathbb E\|F_1(0)\|^2]ds\notag\\
\leq  &  12 \iota_\theta^2 c_{F_1}^2 \kappa^2 t^{1+2(\theta-\eta)}\int_0^t (t-s)^{-2\theta} s^{2(\beta-\eta)} ds+\frac{12 \iota_\theta^2 \mathbb E\|F_1(0)\|^2}{1-2\theta} t^{2(1-\beta)}\notag\\
= & 12\iota_\theta^2 c_{F_1}^2 \kappa^2 B( 1+2\beta-2\eta, 1-2\theta) t^{2(1+\beta-2\eta)}+\frac{12 \iota_\theta^2 \mathbb E\|F_1(0)\|^2}{1-2\theta} t^{2(1-\beta)}.
\end{align*}
Thus,
\begin{align}
&t^{2(\theta-\beta)}\mathbb E\|A^\theta\{\Phi Y\}(t)\|^2  \notag\\
\leq  &  3\iota_{\theta-\beta}^2 \mathbb E\|A^\beta \xi\|^2+6 \iota_\theta^2 \|F_2\|_{\mathcal F^{\beta,\sigma}}^2 B(\beta, 1-\theta)^2  +\frac{3 \iota_\theta^2  \|G\|_{B([0,T];L_2(U;H))}^2   t^{1-2\beta}}{1-2\theta}\notag\\
&+ 12\iota_\theta^2 c_{F_1}^2 \kappa^2 B( 1+2\beta-2\eta, 1-2\theta) t^{2(1+\beta-2\eta)}+\frac{12 \iota_\theta^2 \mathbb E\|F_1(0)\|^2}{1-2\theta} t^{2(1-\beta)}.  \notag
\end{align}

We apply these estimates with  $\theta=\eta$ and $\theta=\beta$. It is therefore observed that if $C_1$ and $C_2$ are fixed in such a way that
\begin{equation} \label{P18}
\begin{aligned}
C_1>&3\iota_{\eta-\beta}^2 \mathbb E\|A^\beta \xi\|^2+6 \iota_\eta^2 \|F_2\|_{\mathcal F^{\beta,\sigma}}^2 B(\beta, 1-\eta)^2,\\
C_2>&3\iota_0^2 \mathbb E\|A^\beta \xi\|^2+6 \iota_\beta^2 \|F_2\|_{\mathcal F^{\beta,\sigma}}^2 B(\beta, 1-\beta)^2,\\
\end{aligned}
\end{equation}
and if   $S$ is sufficiently small,    then 
\begin{align}
t^{2(\eta-\beta)}  & \mathbb E\|A^\eta\{\Phi Y\}(t)\|^2        \notag\\
  \leq & C_1+\frac{3 \iota_\eta^2  \|G\|_{B([0,T];L_2(U;H))}^2   t^{1-2\beta}}{1-2\eta}\notag\\
  &+ 12\iota_\eta^2 c_{F_1}^2 \kappa^2 B( 1+2\beta-2\eta, 1-2\eta) t^{2(1+\beta-2\eta)}+\frac{12 \iota_\eta^2 \mathbb E\|F_1(0)\|^2}{1-2\eta} t^{2(1-\beta)}\notag\\
 \leq& \frac{\kappa^2}{2}+\frac{3 \iota_\eta^2  \|G\|_{B([0,T];L_2(U;H))}^2   t^{1-2\beta}}{1-2\eta}\notag\\
 &+ 12\iota_\eta^2 c_{F_1}^2 \kappa^2 B( 1+2\beta-2\eta, 1-2\eta) t^{2(1+\beta-2\eta)}
          +\frac{12 \iota_\eta^2 \mathbb E\|F_1(0)\|^2}{1-2\eta} t^{2(1-\beta)}\notag\\
<&\kappa^2, \hspace{3cm} 0<t\leq S,\label{P19}
\end{align}
and 
\begin{align}
\mathbb E  &  \|A^\beta\{\Phi Y\}(t)\|^2             \notag\\
  \leq& C_2+\frac{3 \iota_\beta^2  \|G\|_{B([0,T];L_2(U;H))}^2   t^{1-2\beta}}{1-2\beta}\notag\\
 & + 12\iota_\beta^2 c_{F_1}^2 \kappa^2 B( 1+2\beta-2\eta, 1-2\beta) t^{2(1+\beta-2\eta)}\notag\\
  &+\frac{12 \iota_\beta^2 \mathbb E\|F_1(0)\|^2}{1-2\beta} t^{2(1-\beta)}\notag\\
 \leq&\frac{\kappa^2}{2}+\frac{3 \iota_\beta^2  \|G\|_{B([0,T];L_2(U;H))}^2   t^{1-2\beta}}{1-2\beta}\notag\\
 & +12\iota_\beta^2 c_{F_1}^2 \kappa^2 B( 1+2\beta-2\eta, 1-2\beta) t^{2(1+\beta-2\eta)}\notag\\
  &+\frac{12 \iota_\beta^2 \mathbb E\|F_1(0)\|^2}{1-2\beta} t^{2(1-\beta)}\notag\\
<&\kappa^2, \hspace{3cm} 0<t\leq S. \label{P20}
\end{align}
 We have thus shown that
 \begin{equation*} 
\max\left\{\sup_{0<t\leq S} t^{2(\eta-\beta)} \mathbb E\|A^\eta \Phi Y(t)\|^2,
 \sup_{0\leq t\leq S}\mathbb E\|A^\beta \Phi Y(t)\|^2\right\} \leq \kappa^2.
\end{equation*}
 This means that $ \Phi Y$ satisfies \eqref{P15}.  
 
 Next, we  prove that 
 $$\Phi Y\in  \mathcal C((0,S];\mathcal D(A^\eta))\cap \mathcal C([0,S];\mathcal D(A^\beta)) \hspace{1cm} \text{a.s.}$$
Divide $\Phi Y$ into two parts: $\Phi Y(t)=\Psi Y(t)+ W_G(t),$
where 
\begin{equation}  \label{P21}
{\Psi Y}(t)=S(t) \xi+\int_0^t S(t-s)[F_1(Y(s))+F_2(s)]ds,
\end{equation}
and $W_G$ is the stochastic convolution defined in Lemma \ref{Thm4}. 
 Lemma \ref{Thm4}-{\rm (ii)} for $\theta=\eta$  provides that  
$$W_G\in  \mathcal C([0,S];\mathcal D(A^\eta))\subset \mathcal C([0,S];\mathcal D(A^\beta))  \hspace{1cm} \text{a.s.}$$
Therefore,  it suffices to verify that 
\begin{equation}  \label{P15.5}
\Psi Y\in  \mathcal C((0,S];\mathcal D(A^\eta))\cap \mathcal C([0,S];\mathcal D(A^\beta)) \hspace{1cm} \text{a.s.}
\end{equation}

In order to  prove \eqref{P15.5}, we use the Kolmogorov test.  
For $0<s<t\leq S$,  the semigroup property  gives 
\begin{align*}
&\Psi Y(t)-\Psi Y(s)\\
=&S(t-s)S(s)\xi+ S(t-s)\int_0^s S(s-r)[F_1(Y(r))+F_2(r)]dr\\
& +\int_s^t S(t-r)[F_1(Y(r))+F_2(r)]dr -\Psi Y(s)\\
=&[S(t-s)-I]\Psi Y(s)+\int_s^t S(t-r)[F_1(Y(r))+F_2(r)]dr.
\end{align*}
Let $\frac{1}{2}<\rho< 1-\eta$. Thanks to  \eqref{P5},  \eqref{P6} and \eqref{P8}, we have
\begin{align*}
&\|A^\eta[\Psi Y(t)-\Psi Y(s)]\|  \notag\\
\leq &\|[S(t-s)-I]A^{-\rho}\|  \|A^{\eta+\rho}\Psi Y(s)\|\notag\\
&+\int_s^t \|A^\eta S(t-r)\| [\|F_1(Y(r))\|+\|F_2(r)\|]dr \notag\\
\leq &\frac{\iota_{1-\rho}(t-s)^\rho}{\rho} \Big|\Big|A^{\eta+\rho}\Big[S(s) \xi+\int_0^s S(s-r)[F_1(Y(r))+F_2(r)]dr\Big]\Big|\Big| \notag \\
&+\iota_\eta\int_s^t (t-r)^{-\eta} [\|F_1(Y(r))\|+\|F_2(r)\|]dr \notag\\
\leq &\frac{\iota_{1-\rho}(t-s)^\rho}{\rho}  \|A^{\eta+\rho-\beta}S(s)\| \|A^\beta  \xi\|\notag\\
&+\frac{\iota_{1-\rho}(t-s)^\rho}{\rho} \int_0^s \|A^{\eta+\rho} S(s-r)\| \|F_1(Y(r))\|dr \notag \\
&+\frac{\iota_{1-\rho}(t-s)^\rho}{\rho} \int_0^s \|A^{\eta+\rho} S(s-r)\| \|F_2(r)\|dr\notag\\
&+\iota_\eta\int_s^t (t-r)^{-\eta} \|F_1(Y(r))\|dr+\iota_\eta\int_s^t (t-r)^{-\eta} \|F_2(r)\|dr \notag\\
\leq &\frac{\iota_{1-\rho}\iota_{\eta+\rho-\beta}(t-s)^\rho}{\rho}  s^{-\eta-\rho+\beta} \|A^\beta  \xi\|\notag\\
&+\frac{\iota_{1-\rho}\iota_{\eta+\rho} (t-s)^\rho}{\rho}\int_0^s  (s-r)^{-\eta-\rho} \|F_1(Y(r))\|dr \notag\\
&+\frac{\iota_{1-\rho}\iota_{\eta+\rho} \|F_2\|_{\mathcal F^{\beta,\sigma}}(t-s)^\rho}{\rho} \int_0^s (s-r)^{-\eta-\rho}   r^{\beta-1} dr \notag \\
& +\iota_\eta\int_s^t (t-r)^{-\eta} \|F_1(Y(r))\|dr+\iota_\eta\|F_2\|_{\mathcal F^{\beta,\sigma}}\int_s^t (t-r)^{-\eta} r^{\beta-1} dr \notag\\
= &\frac{\iota_{1-\rho} \iota_{\eta+\rho-\beta}}{\rho}\|A^\beta  \xi\| s^{\beta-\eta-\rho}(t-s)^\rho \notag\\
&+\frac{\iota_{1-\rho} \iota_{\eta+\rho}\|F_2\|_{\mathcal F^{\beta,\sigma}}B(\beta,1-\eta-\rho)}{\rho}  s^{\beta-\eta-\rho}(t-s)^\rho \notag \\
&+\iota_\eta\|F_2\|_{\mathcal F^{\beta,\sigma}}\int_s^t (t-r)^{-\eta}r^{\beta-1}dr \notag\\
&+\frac{\iota_{1-\rho}\iota_{\eta+\rho} (t-s)^\rho}{\rho}\int_0^s  (s-r)^{-\eta-\rho} \|F_1(Y(r))\|dr\notag\\
&+\iota_\eta\int_s^t (t-r)^{-\eta} \|F_1(Y(r))\|dr.\notag 
\end{align*}
Dividing $\beta-1$ as $\beta-1=(\eta+\rho-1)+(\beta-\eta-\rho),$ it follows that 
\begin{align*}
\int_s^t (t-r)^{-\eta}  r^{\beta-1} dr &\leq \int_s^t (t-r)^{-\eta}  (r-s)^{\eta+\rho-1} s^{\beta-\eta-\rho} dr\\
&= B(\eta+\rho,1-\eta) s^{\beta-\eta-\rho}(t-s)^\rho.
\end{align*}
Hence, 
\begin{align*}
&\|A^\eta[\Psi Y(t)-\Psi Y(s)]\|  \notag\\
\leq  &\frac{\iota_{1-\rho}\iota_{\eta+\rho-\beta}}{\rho} \|A^\beta  \xi\| s^{\beta-\eta-\rho}(t-s)^\rho \notag\\
&+\Big[\frac{\iota_{1-\rho} \iota_{\eta+\rho} B(\beta,1-\eta-\rho)}{\rho} +\iota_\eta B(\eta+\rho,1-\eta)\Big]\|F_2\|_{\mathcal F^{\beta,\sigma}}s^{\beta-\eta-\rho}(t-s)^\rho \notag\\
&+\frac{\iota_{1-\rho}\iota_{\eta+\rho}}{\rho} (t-s)^\rho \int_0^s  (s-r)^{-\eta-\rho} \|F_1(Y(r))\|dr \notag\\
& +\iota_\eta\int_s^t (t-r)^{-\eta} \|F_1(Y(r))\|dr.\notag 
\end{align*}
Taking the expectation of the squares of the both hand sides of the above inequality, we obtain that 
\begin{align*}
\mathbb E&\|A^\eta[\Psi Y(t)-\Psi Y(s)]\|^2 \notag\\
\leq &\frac{4\iota_{1-\rho}^2\iota_{\eta+\rho-\beta}^2}{\rho^2}  \mathbb E\|A^\beta  \xi\|^2 s^{2(\beta-\eta-\rho)}(t-s)^{2\rho} \notag\\
& +4\Big[\frac{\iota_{1-\rho}\iota_{\eta+\rho} B(\beta,1-\eta-\rho)}{\rho} +\iota_\eta B(\eta+\rho,1-\eta)\Big]^2\notag\\
&\times \|F_2\|_{\mathcal F^{\beta,\sigma}}^2s^{2(\beta-\eta-\rho)}(t-s)^{2\rho} \notag\\
&+\frac{4\iota_{1-\rho}^2\iota_{\eta+\rho}^2}{\rho^2} (t-s)^{2\rho}\mathbb E\Big[ \int_0^s  (s-r)^{-\eta-\rho} \|F_1(Y(r))\|dr\Big]^2  \notag\\
&+4\iota_\eta^2\mathbb E\Big[\int_s^t (t-r)^{-\eta} \|F_1(Y(r))\|dr\Big]^2.\notag
\end{align*}
Since 
\begin{align*}
&\Big[ \int_0^s  (s-r)^{-\eta-\rho} \|F_1(Y(r))\|dr\Big]^2  \\
&=\Big[ \int_0^s  (s-r)^{\frac{-\eta-\rho}{2}} (s-r)^{\frac{-\eta-\rho}{2}}\|F_1(Y(r))\|dr\Big]^2  \\
&\leq \int_0^s  (s-r)^{-\eta-\rho}dr \int_0^s (s-r)^{-\eta-\rho}\|F_1(Y(r))\|^2dr\\
&=\frac{s^{1-\eta-\rho}}{1-\eta-\rho} \int_0^s (s-r)^{-\eta-\rho}\|F_1(Y(r))\|^2dr,
\end{align*}
we arrive at 
\begin{align}
\mathbb E&\|A^\eta[\Psi Y(t)-\Psi Y(s)]\|^2 \label{P22}\\
\leq &\frac{4\iota_{1-\rho}^2  \iota_{\eta+\rho-\beta}^2}{\rho^2} \mathbb E\|A^\beta  \xi\|^2 s^{2(\beta-\eta-\rho)}(t-s)^{2\rho} \notag\\
& +4\Big[\frac{\iota_{1-\rho}\iota_{\eta+\rho} B(\beta,1-\eta-\rho)}{\rho}  +\iota_\eta B(\eta+\rho,1-\eta)\Big]^2 \notag\\
&\times\|F_2\|_{\mathcal F^{\beta,\sigma}}^2s^{2(\beta-\eta-\rho)}(t-s)^{2\rho} \notag\\
&+\frac{4\iota_{1-\rho}^2\iota_{\eta+\rho}^2}{\rho^2(1-\eta-\rho)} (t-s)^{2\rho} s^{1-\eta-\rho}     \int_0^s  (s-r)^{-\eta-\rho} \mathbb E\|F_1(Y(r))\|^2dr  \notag\\
&+4\iota_\eta^2(t-s) \int_s^t (t-r)^{-2\eta}\mathbb E \|F_1(Y(r))\|^2dr.\notag
\end{align}

Both the integrals in \eqref{P22} can be estimated by using   \eqref{P17}:   
\begin{align}
& \int_0^s  (s-r)^{-\eta-\rho} \mathbb E\|F_1(Y(r))\|^2dr \notag\\
 \leq &
2c_{F_1}^2 \kappa^2\int_0^s  (s-r)^{-\eta-\rho}  r^{2(\beta-\eta)} dr  +2\mathbb E\|F_1(0)\|^2\int_0^s  (s-r)^{-\eta-\rho} dr  \notag\\
 = &
2c_{F_1}^2 \kappa^2 B(1+2\beta-2\eta,1-\eta-\rho) s^{1+2\beta-3\eta-\rho}  \label{P23}\\
&+\frac{2\mathbb E\|F_1(0)\|^2 s^{1-\eta-\rho}}{1-\eta-\rho}, \notag
\end{align}
and
\begin{align}
&\int_s^t (t-r)^{-2\eta}\mathbb E \|F_1(Y(r))\|^2dr\notag\\
&\leq 2\int_s^t (t-r)^{-2\eta}[c_{F_1}^2 \kappa^2 r^{2(\beta-\eta)} + \mathbb E\|F_1(0)\|^2]dr\notag\\
&=2 c_{F_1}^2 \kappa^2 \int_s^t (t-r)^{-2\eta}r^{2(\beta-\eta)} dr+\frac{2\mathbb E\|F_1(0)\|^2}{1-2\eta} (t-s)^{1-2\eta}.\label{P24}
\end{align}
Divide  $2(\beta-\eta)$ as $2(\beta-\eta)=(\beta-\frac{1}{2})+(\frac{1}{2}+\beta-2\eta).$  Then 
\begin{align}
&\int_s^t (t-r)^{-2\eta} r^{2(\beta-\eta)} dr\notag\\
&\leq \int_s^t (t-r)^{-2\eta} (r-s)^{\beta-\frac{1}{2}} t^{\frac{1}{2}+\beta-2\eta}dr\notag\\
&=B(\frac{1}{2}+\beta,1-2\eta) t^{\frac{1}{2}+\beta-2\eta} (t-s)^{\frac{1}{2}+\beta-2\eta}.  \label{P25}
\end{align}
Combining \eqref{P22},  \eqref{P23},  \eqref{P24} and  \eqref{P25}, we obtain an estimate:
\begin{align}
\mathbb E&\|A^\eta[\Psi Y(t)-\Psi Y(s)]\|^2 \label{P26}\\
\leq  &\frac{4\iota_{1-\rho}^2  \iota_{\eta+\rho-\beta}^2}{\rho^2} \mathbb E\|A^\beta  \xi\|^2 s^{2(\beta-\eta-\rho)}(t-s)^{2\rho} \notag\\
& +4\Big[\frac{\iota_{1-\rho}\iota_{\eta+\rho} B(\beta,1-\eta-\rho)}{\rho}  +\iota_\eta B(\eta+\rho,1-\eta)\Big]^2\notag\\
&\times \|F_2\|_{\mathcal F^{\beta,\sigma}}^2s^{2(\beta-\eta-\rho)}(t-s)^{2\rho}  \notag\\
&+\frac{8\iota_{1-\rho}^2\iota_{\eta+\rho}^2c_{F_1}^2 \kappa^2 B(1+2\beta-2\eta,1-\eta-\rho)}{\rho^2(1-\eta-\rho)} (t-s)^{2\rho} s^{2(1+\beta-2\eta-\rho)}  \notag\\
&+\frac{8\iota_{1-\rho}^2\iota_{\eta+\rho}^2\mathbb E\|F_1(0)\|^2}{\rho^2(1-\eta-\rho)^2} (t-s)^{2\rho} s^{2(1-\eta-\rho)}  \notag\\
&+8 \iota_\eta^2c_{F_1}^2 \kappa^2 B(\frac{1}{2}+\beta,1-2\eta) t^{\frac{1}{2}+\beta-2\eta} (t-s)^{\frac{3}{2}+\beta-2\eta}\notag\\
&+\frac{8\iota_\eta^2\mathbb E\|F_1(0)\|^2}{1-2\eta} (t-s)^{2(1-\eta)}, \hspace{2cm} 0<s<t\leq S.  
\notag
\end{align}

Since this estimate holds true for any $\frac{1}{2}<\rho< 1-\eta,$ and since  $1<\frac{3}{2}+\beta-2\eta<2(1-\eta)$, Theorem \ref{Thm1}  then provides that $A^\eta \Psi Y $ is H\"older continuous on $(0,S]$ with an arbitrarily smaller exponent  than $\frac{1+2\beta}{4}-\eta.$ As a consequence,  for any $ 0< \gamma < \frac{1+2\beta}{4}-\eta$, 
\begin{equation}\label{P27}
\begin{cases}
\Psi Y\in \mathcal C((0,S];\mathcal D(A^\eta))\subset \mathcal C((0,S];\mathcal D(A^\beta)) \hspace{1cm} \text{a.s.,}\\
A^\eta \Psi Y\in \mathcal C^\gamma((0,S];H) \hspace{2cm} \text{a.s.}
\end{cases}
\end{equation}

In view of \eqref{P15.5} and \eqref{P27}, it  remains to    show that $A^\beta \Psi Y$ is continuous at $t=0$. This function is separated into three terms: 
$$A^\beta\Psi Y(t)=A^\beta S(t) \xi+A^\beta \int_0^t S(t-s)F_2(s)ds+A^\beta \int_0^t S(t-s)F_1(Y(s))ds.$$
Obviously, the first term $A^\beta S(\cdot) \xi$ is  continuous at $t=0$, since 
$$\lim_{t\to 0}\|A^\beta S(t)\xi-A^\beta \xi\|=\lim_{t\to 0}\|[S(t)-I]A^\beta \xi\|=0.$$

The continuity of the second term $A^\beta \int_0^\cdot S(\cdot-s)F_2(s)ds$ at $t=0$ is verified in the following way.
By the property of the space $\mathcal F^{\beta,\sigma}((0,T];H),$  we may put $z=\lim_{t\to 0} t^{1-\beta}F_2(t)$.  Then, 
\begin{align*}
\Big|\Big|A^\beta & \int_0^t S(t-s)F_2(s)ds\Big|\Big|\notag\\
\leq &\Big|\Big|\int_0^t A^\beta S(t-s)[F_2(s)-F_2(t)]ds\Big|\Big|+\Big|\Big|\int_0^t A^\beta S(t-s)F_2(t)ds\Big|\Big|\notag\\
=&\Big|\Big|\int_0^t A^\beta S(t-s)[F_2(s)-F_2(t)]ds\Big|\Big|+\Big|\Big|[I-S(t)]A^{\beta-1}F_2(t)\Big|\Big|\notag\\
\leq &\int_0^t \|A^\beta S(t-s)\| \|F_2(t)-F_2(s)\|ds\notag\\
&
+\| t^{\beta-1}[I-S(t)]A^{\beta-1} [t^{1-\beta}F_2(t)-z]\|+\| t^{\beta-1}[I-S(t)]A^{\beta-1} z\|.\notag
\end{align*}
Thereby,  \eqref{P4}, \eqref{P6} and \eqref{P8} give 
\begin{align*}
\limsup_{t\to 0}&\Big|\Big|A^\beta\int_0^t S(t-s)F_2(s)ds\Big|\Big|\notag\\
\leq & \iota_\beta \limsup_{t\to 0}\int_0^t \iota_\beta (t-s)^{-\beta}\|F_2(t)-F_2(s)\|ds\notag\\
&+\frac{\iota_\beta}{1-\beta}\limsup_{t\to 0}\|t^{1-\beta}F_2(t)-z\|\notag\\
&+\limsup_{t\to 0}\| t^{\beta-1}[I-S(t)]A^{\beta-1} z\|\notag\\
= & \iota_\beta \limsup_{t\to 0}\int_0^t (t-s)^{\sigma-\beta}s^{-1+\beta-\sigma}  \frac{ s^{1-\beta+\sigma}\|F_2(t)-F_2(s)\|}{(t-s)^\sigma}ds\notag\\
&+\frac{\iota_\beta}{1-\beta}\limsup_{t\to 0}\|t^{1-\beta}F_2(t)-z\|+\limsup_{t\to 0}\| t^{\beta-1}[I-S(t)]A^{\beta-1} z\|\notag\\
\leq &\iota_\beta B(\beta-\sigma, 1-\beta+\sigma)\limsup_{t\to 0} \sup_{s\in[0,t)}\frac{ s^{1-\beta+\sigma}\|F_2(t)-F_2(s)\|}{(t-s)^\sigma}\notag\\
&+\limsup_{t\to 0}\| t^{\beta-1}[I-S(t)]A^{\beta-1} z\|\notag\\
=&\limsup_{t\to 0}\| t^{\beta-1}[I-S(t)]A^{\beta-1} z\|. \notag
\end{align*}
Since $\mathcal D(A^\beta)$ is dense in $H$, there exists a sequence $\{z_n\}_n$  in $\mathcal D(A^\beta)$ that converges to $z$ as $n\to \infty.$ Hence, \eqref{P8} gives 
\begin{align*}
&\limsup_{t\to 0}\Big|\Big|A^\beta\int_0^t S(t-s)F_2(s)ds\Big|\Big|\\
&\leq\limsup_{t\to 0}  \|t^{\beta-1}[I-S(t)]A^{\beta-1} (z-z_n)\|\notag\\
&+\limsup_{t\to 0}  \|t^{\beta-1}[I-S(t)]A^{-1}A^{\beta} z_n\|\\
&\leq \frac{\iota_\beta}{1-\beta}\|z-z_n\|+\iota_0 \limsup_{t\to 0}  t^\beta \|A^{\beta} z_n\|\\
&=  \frac{\iota_\beta}{1-\beta}\|z-z_n\|, \hspace{2cm} n=1,2,\dots
\end{align*}
Letting $n$ to $\infty$, we obtain that 
$$\lim_{t\to 0}A^\beta\int_0^t S(t-s)F_2(s)ds=0.$$
 This means that  $A^\beta\int_0^\cdot S(\cdot-s)F_2(s)ds$ is continuous at $t=0$.

To see the continuity of the last term $A^\beta \int_0^\cdot S(\cdot-s)F_1(Y(s))ds$ at $t=0$, 
 using \eqref{P6} and \eqref{P17}, we have
\begin{align} 
&\mathbb E\Big|\Big|A^\beta \int_0^t S(t-s)F_1(Y(s))ds\Big|\Big|^2 \notag\\
 & \leq\mathbb E\Big[\int_0^t \|A^\beta S(t-s)\| \|F_1(Y(s))\|ds\Big]^2\notag\\
 &\leq \iota_\beta^2 \mathbb E  \Big[\int_0^t (t-s)^{-\beta}\|F_1(Y(s))\| ds\Big]^2\notag\\
 &\leq \iota_\beta^2 t   \int_0^t (t-s)^{-2\beta}\mathbb E\|F_1(Y(s))\|^2 ds\notag\\
 &\leq 2\iota_\beta^2 t   \int_0^t (t-s)^{-2\beta}[c_{F_1}^2 \kappa^2 s^{2(\beta-\eta)} + \mathbb E\|F_1(0)\|^2] ds\notag\\
 &= 2\iota_\beta^2 \Big[c_{F_1}^2 \kappa^2 B(1+2\beta-2\eta,1-2\beta)t^{2(1-\eta)} +\frac{\mathbb E\|F_1(0)\|^2}{1-2\beta} t^{2(1-\beta)}\Big]  \label{P28}\\ 
 & \to 0 \hspace{2cm} \text{   as   } t\to 0. \notag
\end{align}
Therefore, there exists a decreasing sequence $\{t_n\}_{n=1}^\infty$ converging to $0$ such that 
$$\lim_{n\to\infty} A^\beta \int_0^{t_n} S(t_n-s)F_1(Y(s))ds=0.$$
Since $A^\beta \int_0^\cdot S(\cdot-s)F_1(Y(s))ds$ is continuous on $(0,S]$, we conclude that 
$$\lim_{t\to 0} A^\beta \int_0^t S(t-s)F_1(Y(s))ds=0,$$
i.e. $A^\beta \int_0^\cdot S(\cdot-s)F_1(Y(s))ds$ is continuous at $t=0$.

{\bf Step 2}. Let us show that $\Phi$ is a contraction mapping of $\Xi (S)$, provided that  $S>0$ is sufficiently small. 

Let $Y_1,Y_2\in \Xi (S)$ and $0\leq \theta< \frac{1}{2}.$ It follows from \eqref{P16} that 
\begin{align*}
&t^{2(\theta-\beta)}\mathbb E\|A^\theta[\Phi Y_1(t)-\Phi Y_2(t)]\|^2  \notag \\
=&t^{2(\theta-\beta)} \mathbb E\Big|\Big|\int_0^t A^\theta S(t-s)[F_1(Y_1(s))-F_1(Y_2(s))]ds\Big|\Big|^2 \notag\\
\leq&t^{2(\theta-\beta)} \mathbb E\Big[\int_0^t \|A^\theta S(t-s)\| \|F_1(Y_1(s))-F_1(Y_2(s))\|ds\Big]^2. \notag
\end{align*}
Hence,  \eqref{P6},  {\rm (H3)} and  \eqref{P13} give 
\begin{align*}
&t^{2(\theta-\beta)}\mathbb E\|A^\theta[\Phi Y_1(t)-\Phi Y_2(t)]\|^2  \notag \\
\leq& c_{F_1}^2\iota_\theta^2 t^{2(\theta-\beta)} \mathbb E\Big[\int_0^t (t-s)^{-\theta}   \|A^\eta(Y_1(s)-Y_2(s))\|ds\Big]^2 \notag\\
\leq& c_{F_1}^2\iota_\theta^2 t^{1+2(\theta-\beta)} \mathbb E\int_0^t (t-s)^{-2\theta}   \|A^\eta(Y_1(s)-Y_2(s))\|^2ds \notag\\
\leq& c_{F_1}^2\iota_\theta^2 t^{1+2(\theta-\beta)} \int_0^t (t-s)^{-2\theta}  \mathbb E \|A^\eta(Y_1(s)-Y_2(s))\|^2ds \notag\\
\leq&c_{F_1}^2\iota_\theta^2 t^{1+2(\theta-\beta)} \int_0^t  (t-s)^{-2\theta} s^{2(\beta-\eta)} \|Y_1-Y_2\|_{{\Xi (S)}}^2ds \notag\\
=&c_{F_1}^2\iota_\theta^2  B(1+2\beta-2\eta,1-2\theta) t^{2(1-\eta)}   \|Y_1-Y_2\|_{{\Xi (S)}}^2. \notag 
\end{align*}
Applying these estimates with $\theta=\eta$ and $\theta=\beta$, we conclude that
\begin{align}
&\|\Phi Y_1-\Phi Y_2\|_{{\Xi (S)}}^2 \notag\\
=&\sup_{0<t\leq S} t^{2(\eta-\beta)} \mathbb E\|A^\eta [\Phi Y_1(t)-\Phi Y_2(t)]\|^2\notag\\
&+ \sup_{0\leq t\leq S}\mathbb E\|A^\beta  [\Phi Y_1(t)-\Phi Y_2(t)]\|^2\notag\\
\leq &c_{F_1}^2  [\iota_\eta^2B(1+2\beta-2\eta,1-2\eta)+\iota_\beta^2B(1+2\beta-2\eta,1-2\beta)] \label{P29}\\
&\times S^{2(1-\eta)} \|Y_1-Y_2\|_{{\Xi (S)}}^2. \notag
\end{align}
Clearly, \eqref{P29}  shows that $\Phi$ is contractive in $\Xi (S),$ provided that  $S>0$ is sufficiently small.

{\bf Step 3}.  Let us prove:
\begin{itemize}
  \item the existence of a local mild solution in the function space in \eqref{P10}
  \item the estimate \eqref{P11}
\end{itemize}

Let $S>0$ be sufficiently small in such a way that $\Phi$ maps $\Upsilon(S)$ into itself and is contraction with respect to the norm of $\Xi (S).$ Due to  Step 1 and  Step 2, $S=T_{loc}$ can be determined by $\mathbb E\|F_1(0)\|^2,$ $ \|F_2\|_{\mathcal F^{\beta,\sigma}}^2,$ $ \|G\|_{B([0,T];L_2(U;H))}^2$ and $\mathbb E \|A^\beta\xi\|^2.$ Thanks to the fixed point theorem, there exists a unique function $X\in \Upsilon(T_{loc})$ such that $X=\Phi X$. 
This means that $X$ is a local mild solution of  \eqref{P1} in the function space:
$$
X\in  \mathcal C((0,T_{loc}];\mathcal D(A^\eta))\cap \mathcal C([0,T_{loc}];\mathcal D(A^\beta)) \hspace{1cm} \text{a.s.}
$$
In addition, thanks to Lemma \ref{Thm4}-{\rm(iii)} and  \eqref{P27}, for any $ 0<\gamma < \frac{1+2\beta}{4}-\eta,$
$$A^\eta X=A^\eta \Phi X=A^\eta \Psi X +A^\eta W_G \in  \mathcal C^\gamma((0,T_{loc}];H)  \hspace{1cm} \text{a.s.}$$
Furthermore,  \eqref{P11} is obtained from the definition of $\Upsilon(T_{loc}) $ (see \eqref{P15}).

{\bf Step 4}.  Let us finally show the uniqueness of local mild solutions. 

Let $\bar X$ be any other local mild solution to  \eqref{P1} on the interval $[0,T_{loc}],$ which belongs to the space $\mathcal C((0,T_{loc}];\mathcal D(A^\eta))\cap \mathcal C([0,T_{loc}];\mathcal D(A^\beta))$.

The formulae
\begin{align} 
 X(t)=& S(t) \xi+\int_0^t S(t-s)F_2(s)ds + \int_0^t S(t-s)G(s) dW(s)    \label{P30}  \\
&+ \int_0^t S(t-s)F_1(X(s))ds, \notag
\end{align}
and
\begin{align*} 
\bar X(t)=& S(t) \xi+\int_0^t S(t-s)F_2(s)ds + \int_0^t S(t-s)G(s) dW(s) \\
&+ \int_0^t S(t-s)F_1(\bar X(s))ds
\end{align*}
 imply that
$$X(t)-\bar X(t)=\int_0^t S(t-s)[F_1(X(s))-F_1(\bar X(s))] ds, \quad\quad 0\leq t\leq T_{loc}.$$
We can then repeat the same arguments as in  Step 2  to deduce that
\begin{align}
&\|X-\bar X\|_{\Xi (\bar T)}^2 \label{P31} \\
\leq & c_{F_1}^2  [\iota_\eta^2B(1+2\beta-2\eta,1-2\eta)+\iota_\beta^2B(1+2\beta-2\eta,1-2\beta)] \notag\\
&\times  {\bar T}^{2(1-\eta)} \|X-\bar X\|_{\Xi (\bar T)}^2\notag
\end{align}
   for any   $ 0< \bar T\leq T_{loc}$. 
Let $\bar T$ be a positive constant such that
\begin{align*}
 &c_{F_1}^2  [\iota_\eta^2B(1+2\beta-2\eta,1-2\eta)+\iota_\beta^2B(1+2\beta-2\eta,1-2\beta)]\bar T^{2(1-\eta)}<1.
\end{align*}
Thus,    \eqref{P31} gives
 $$X(t)=\bar X(t) \hspace{1cm} \text{ a.s., } 0 \leq t \leq \bar T.$$

 We repeat the same procedure with  initial time $\bar T$ and  initial value $X(\bar T)=\bar X(\bar T)$ to derive that 
$$X(\bar T +t)=\bar X(\bar T +t) \hspace{1cm} \text{ a.s., } 0 \leq t \leq \bar T.$$
 This means that $X(t)=\bar X(t)$ a.s. on a larger interval $[0,2\bar T].$ We continue this procedure by finite times, the extended interval can cover the given interval $[0,T_{loc}].$ Therefore,  for  $0\leq t\leq T_{loc},$ $X(t)=\bar X(t)$ a.s.
\end{proof}

Let us next show the differentiability  of the expectation of local mild solutions. Put 
$$Z(t)=\mathbb E X(t), \hspace{1cm} 0\leq t \leq T_{loc}.$$
\begin{theorem}  \label{Thm7}
Let the assumptions in Theorem \ref{Thm5} be satisfied. Assume that $\sigma+\eta\leq \frac{1}{2}.$ Then, 
\begin{equation} \label{P33}
\begin{cases}
Z\in  \mathcal C((0,T_{loc}];\mathcal D(A))\cap \mathcal C([0,T_{loc}];\mathcal D(A^\beta))\cap \mathcal C^1((0,T_{loc}];H),\\
\frac{dZ}{dt}, AZ \in \mathcal F^{\beta,\sigma}((0,T_{loc}];H).
\end{cases}
\end{equation}
 Furthermore, $Z$ satisfies the estimate
\begin{equation}\label{P34}
\|A^\beta Z\|_{\mathcal C} + \Big|\Big|\frac{dZ}{dt} \Big|\Big|_{\mathcal F^{\beta,\sigma}}+ \|AZ\|_{\mathcal F^{\beta,\sigma}}\leq C_{F_1,F_2,\xi},    \hspace{2cm} 0\leq t\leq T_{loc}
\end{equation}
with some constant  $C_{F_1,F_2,\xi}$ depending  on   $\mathbb E \|F_1(0)\|^2,$ $ \mathbb E \|A^\beta\xi\|^2$,  $\|F_2\|_{\mathcal F^{\beta,\sigma}}^2$ and  $T_{loc}.$
\end{theorem}
\begin{proof}
Throughout the proof, we use a universal constant $C$, which depends on exponents and $\mathbb E \|F_1(0)\|^2,$ $ \mathbb E \|A^\beta\xi\|^2$, $\|F_2\|_{\mathcal F^{\beta,\sigma}}^2$ and $T_{loc}$.

Since $\mathbb E\int_0^t S(t-s)G(s)dW(s)=0,$ we have an expression:
$$Z(t)=\mathbb E S(t)\xi + \int_0^t S(t-s) [\mathbb EF_1(X(s))+F_2(s)]ds.$$

First, 
let us show that 
\begin{equation}  \label{P35}
\mathbb EF_1(X(\cdot)) \in \mathcal F^{\beta,\sigma}((0,T_{loc}];H).
\end{equation}
In view of  \eqref{P17}, 
\begin{align*}
\|\mathbb EF_1(X(t))\|^2 \leq \mathbb E\|F_1(X(t))\|^2 \leq C[
 t^{2(\beta-\eta)}  +1], \hspace{1cm} 0<t\leq T_{loc}.
 \end{align*}
Thereby, 
 \begin{align*}
\|t^{1-\beta}\mathbb EF_1(X(t))\|^2  \leq C[
 t^{2(1-\eta)}  +t^{2(1-\beta)}] \to 0 \hspace{1cm} \text{ as } t\to 0.
\end{align*}
The function $\mathbb EF_1(X(\cdot))$  therefore satisfies  \eqref{P2}.

On the other hand,  {\rm (H3)}, \eqref{P21}, \eqref{P26} and \eqref{P30} give  
\begin{align*}
&\|\mathbb EF_1(X(t))-\mathbb EF_1(X(s))\|^2 \\
\leq &\mathbb E\|F_1(X(t))-F_1(X(s))\|^2 \\
\leq & c_{F_1} \mathbb E\|A^{\eta} [X(t)-X(s)]\|^2\\
\leq & C \{\mathbb E\|A^{\eta} [\Psi X(t)-\Psi X(s)]\|^2 + \mathbb E\| A^{\eta} [W_G(t)-W_G(s)] \|^2\}\\
\leq  &C s^{2(\beta-\eta-\rho)}(t-s)^{2\rho} +C s^{2(1+\beta-2\eta-\rho)}  (t-s)^{2\rho} \notag\\
&+C s^{2(1-\eta-\rho)} (t-s)^{2\rho}+C(t-s)^{\frac{3}{2}+\beta-2\eta}+C(t-s)^{2(1-\eta)}\\
&+ \mathbb E\| A^{\eta} [W_G(t)-W_G(s)] \|^2, \hspace {2cm} 0<s<t<T_{loc}
 \end{align*}
for any $\frac{1}{2}<\rho<1-\eta$. The last term in the latter inequality is evaluated by using  \eqref{P6} and \eqref{P8}: 
\begin{align*}
&\mathbb E\| A^{\eta} [W_G(t)-W_G(s)] \|^2\\
 =&\mathbb E\Big|\Big| \int_0^s A^{\eta} [S(t-r)-S(s-r)] G(r) dW(r)+\int_s^t A^{\eta} S(t-r) G(r) dW(r)\Big|\Big|^2\\
\leq & \int_0^s \|A^{\eta+\sigma} S(s-r)[S(t-s)-I]A^{-\sigma} G(r)\|_{L_2(U;H)}^2 dr\\
&+\int_s^t \|A^{\eta} S(t-r) G(r)\|_{L_2(U;H)}^2 dr\\
\leq &\iota_{\eta+\sigma}^2\|[S(t-s)-I]A^{-\sigma}\|^2 \|G\|_{B([0,T];L_2(U;H))}^2 \int_0^s (s-r)^{-2(\eta+\sigma)}  dr\\
&+\iota_\eta^2  \|G\|_{B([0,T];L_2(U;H))}^2\int_s^t (t-r)^{-2\eta}  dr\\
\leq & C s^{1-2(\eta+\sigma)}(t-s)^{2\sigma} + C(t-s)^{1-2\eta}.
\end{align*}
Thus, 
\begin{align*}
&\frac{s^{2(1-\beta+\sigma)}\|\mathbb EF_1(X(t))-\mathbb EF_1(X(s))\|^2}{(t-s)^{2\sigma}}\\
\leq &C s^{2(1-\eta-\rho+\sigma)}(t-s)^{2(\rho-\sigma)} +C s^{2(2+\sigma-2\eta-\rho)}  (t-s)^{2(\rho-\sigma)} \notag\\
&+C s^{2(2-\beta+\sigma-\eta-\rho)} (t-s)^{2(\rho-\sigma)}+Cs^{2(1-\beta+\sigma)}(t-s)^{\frac{3}{2}+\beta-2\eta-2\sigma}\\
&+Cs^{2(1-\beta+\sigma)}(t-s)^{2(1-\eta-\sigma)}+C s^{3-2(\eta+\beta)}\notag \\
&+ Cs^{2(1-\beta+\sigma)} (t-s)^{1-2(\eta+\sigma)}, \hspace{1cm} 0<s<t\leq T_{loc}.
\end{align*}
This shows that \eqref{P3} and \eqref{P4} are also valid for the function $\mathbb EF_1(X(\cdot))$. Hence, \eqref{P35} has been verified.

As a result of  {\rm (H4)} and  \eqref{P35},
$$\mathbb EF_1(X(\cdot))+F_2(\cdot)\in \mathcal F^{\beta,\sigma}((0,T_{loc}];H). $$ 
Since $\mathbb ES(t)\xi=S(t)\mathbb E\xi$ and $\mathbb E\xi\in \mathcal D(A^\beta)$, Theorem \ref{Thm2} applied to the function $Z$  provides   \eqref{P33} and  \eqref{P34}. The proof is completed.
\end{proof}
\subsection{Regular dependence of solutions on initial data} \label{sub3.2}
Let $\mathcal B_1$ and $\mathcal B_2$ be bounded balls:
$$\mathcal B_1=\{f\in \mathcal F^{\beta,\sigma}((0,T];H);  \|f\|_{\mathcal F^{\beta,\sigma}}\leq R_1\}, \quad 0<R_1<\infty,  $$
$$\mathcal B_2=\{g\in B([0,T];L_2(U;H));  \|G\|_{B([0,T];L_2(U;H))}\leq R_2\},  \quad  0<R_2<\infty, $$
 of the spaces $\mathcal F^{\beta,\sigma}((0,T];H)$ and $B([0,T];L_2(U;H))$, respectively.
Let $B_A$ be a set of random variables: 
\begin{equation*}  
B_A=\{\zeta;  \zeta\in \mathcal D(A^\beta) \, \text{   a.s. and    }  \,   \mathbb E \|A^\beta \zeta\|^2\leq R_3^2\}, \quad 0<R_3<\infty.
\end{equation*}

 According to Theorem \ref{Thm5}, for every $F_2\in \mathcal B_1, G\in \mathcal B_2$ and $\xi\in B_A$, there exists a unique local solution of \eqref{P1} on some interval $[0,T_{loc}]$. Furthermore, in view of  Step 1  and  Step 2  in the proof for Theorem 
\ref{Thm5},  we have 
\begin{equation} \label{P36}
\begin{aligned}
&\text{ there is a time    }  T_{\mathcal B_1, \mathcal B_2, B_A}>0 \text{  such that   } \\
&[0,T_{\mathcal B_1, \mathcal B_2, B_A}]\subset [0,T_{loc}]   \,  \text{   for all   } \,  (F_2,G,\xi)\in \mathcal B_1\times \mathcal B_2\times B_A. 
\end{aligned}
\end{equation}
Indeed, by  \eqref{P19}, \eqref{P20} and \eqref{P29}, $T_{loc}$ can be chosen to be any time $S$ satisfying the conditions:
\begin{align*}
\frac{\kappa^2}{2} \geq &\frac{3 \iota_\eta^2  \|G\|_{B([0,T];L_2(U;H))}^2   S^{1-2\beta}}{1-2\eta}\notag\\
&+ 12\iota_\eta^2 c_{F_1}^2 \kappa^2 B( 1+2\beta-2\eta, 1-2\eta) S^{2(1+\beta-2\eta)}\notag\\
 &+\frac{12 \iota_\eta^2 \mathbb E\|F_1(0)\|^2}{1-2\eta} S^{2(1-\beta)},
\end{align*}
\begin{align*}
\frac{\kappa^2}{2}\geq &\frac{3 \iota_\beta^2  \|G\|_{B([0,T];L_2(U;H))}^2   S^{1-2\beta}}{1-2\beta} \\
 &+12\iota_\beta^2 c_{F_1}^2 \kappa^2 B( 1+2\beta-2\eta, 1-2\beta) S^{2(1+\beta-2\eta)}\notag\\
 &+\frac{12 \iota_\beta^2 \mathbb E\|F_1(0)\|^2}{1-2\beta} S^{2(1-\beta)},
\end{align*}
and
\begin{align*}
&1>c_{F_1}^2  [\iota_\eta^2B(1+2\beta-2\eta,1-2\eta)+\iota_\beta^2B(1+2\beta-2\eta,1-2\beta)]   S^{2(1-\eta)},
\end{align*}
where $\kappa$ is defined by \eqref{P14} and \eqref{P18}.  As a consequence, we can choose  $T_{loc}$ such that it depends continuously on  $\mathbb E \|F_1(0)\|^2$,  $ \mathbb E \|A^\beta\xi\|^2$, $ \|G\|_{B([0,T];L_2(U;H))}^2$ and  $\|F_2\|_{\mathcal F^{\beta,\sigma}}^2$. 
  Thus,  \eqref{P36} follows.
 
 We are now ready to show the continuous dependence of solutions on $(F_2,G,\xi)$ in the sense specified in the following theorem.
\begin{theorem}  \label{Thm6}
Let {\rm (H1)}, {\rm (H2)}, {\rm (H3)}, {\rm (H4)} and {\rm (H5)}   be satisfied.
Let $X$ and $\bar X$ be the solutions of  \eqref{P1} for the data $(F_2,G,\xi)$ and $(\bar F_2,\bar G,\bar \xi)$ in $\mathcal B_1\times \mathcal B_2\times B_A$, respectively. Then, there exists a constant $C_{\mathcal B_1, \mathcal B_2, B_A}$ depending only on $\mathcal B_1, \mathcal B_2$ and $ B_A$ such that
\begin{align}
&t^{2\eta}\mathbb E  \|A^\eta[X(t)-\bar X(t)]\|^2+t^{2\eta}\mathbb E\|A^{\beta} [X(t)-\bar X(t)]\|^2 \label{P37}\\
&+ \mathbb E\|X(t)-\bar X(t)\|^2 \leq C_{\mathcal B_1, \mathcal B_2, B_A}[\mathbb E \|\xi-\bar \xi\|^2+ t^{2\beta}   \|F_2-\bar F_2\|_{\mathcal F^{\beta,\sigma}}^2  \notag \\
&+ t \|G-\bar G\|_{B([0,T];L_2(U;H))}^2], \hspace{1cm} 0<t<T_{\mathcal B_1, \mathcal B_2, B_A},\notag
\end{align}
and 
\begin{align}
t^{2(\eta-\beta)}& [\mathbb E  \|A^\eta[X(t)-\bar X(t)]\|^2+\mathbb E\|A^{\beta} [X(t)-\bar X(t)]\|^2] \label{P38}\\
 \leq & C_{\mathcal B_1, \mathcal B_2, B_A}[\mathbb E \|A^\beta(\xi-\bar \xi)\|^2+    \|F_2-\bar F_2\|_{\mathcal F^{\beta,\sigma}}^2\notag \\
&+ \|G-\bar G\|_{B([0,T];L_2(U;H))}^2], \hspace{1cm} 0<t<T_{\mathcal B_1, \mathcal B_2, B_A}. \notag
\end{align}
\end{theorem}

In order to prove this theorem, we  use a generalized inequality of Gronwall type.
\begin{lemma} \label{Thm3}
Let $0<a\leq b,  \mu>0 $ and $\nu>0$ be constants. 
Let $f$ be a continuous and increasing function on $[0,\infty)$ and $\varphi$ be a nonnegative  bounded  function  on $[a,b]$.  If $\varphi$ satisfies the integral inequality 
$$\varphi(t) \leq f(t)+ a^{-\mu} \int_a^t (t-r)^{\nu-1}\varphi(r)dr, \quad\quad a\leq t\leq b,$$
then there exists $c>0$ such that 
$$\varphi(t)\leq c f(t),\hspace{2cm} a\leq s<t\leq b.$$
\end{lemma}

\begin{proof}
Let $\Gamma$ be the gamma function. By induction, we  verify the estimate:
\begin{align}
\varphi(t)\leq &\sum_{k=0}^n a^{-k\mu} f(t)t^{k\nu} \frac{\Gamma(\nu)^k}{\Gamma(1+k\nu)} +a^{-\mu(n+1)}  \frac{\Gamma(\nu)^{n+1}}{\Gamma((n+1)\nu)}  \label{E1}\\
& \times\int_a^t (t-s)^{(n+1)\nu-1} \varphi(s)ds, \hspace{1cm} a\leq t\leq b.   \notag
\end{align}
Indeed, the case $n=0$ is obvious. Assume that this inequality holds true for $n$. Then,
\begin{align}
\varphi(t)\leq &\sum_{k=0}^n a^{-k\mu} f(t)t^{k\nu} \frac{\Gamma(\nu)^k}{\Gamma(1+k\nu)}  \label{E2}\\
&+a^{-\mu(n+1)}  \frac{\Gamma(\nu)^{n+1}}{\Gamma((n+1)\nu)} \int_a^t (t-s)^{(n+1)\nu-1} \notag\\
&\hspace{4cm}\times[f(s)+ a^{-\mu} \int_a^s (s-r)^{\nu-1}\varphi(r)dr]ds.   \notag
\end{align}
Since $f$ is increasing, we observe that 
\begin{align}
\int_a^t (t-s)^{(n+1)\nu-1} f(s)ds &\leq f(t) \frac{(t-a)^{(n+1)\nu}}{(n+1)\nu}\notag\\
&=\frac{f(t)(t-a)^{(n+1)\nu}\Gamma((n+1)\nu)}{\Gamma(1+(n+1)\nu)}.\label{E3}
\end{align}
In addition, 
\begin{align}
&\int_a^t \int_a^s (t-s)^{(n+1)\nu-1} (s-r)^{\nu-1}\varphi(r)drds  \notag\\
&\leq \int_0^t \int_0^s (t-s)^{(n+1)\nu-1} (s-r)^{\nu-1}\varphi(r)drds \notag\\
&= \int_0^t \int_r^t (t-s)^{(n+1)\nu-1} (s-r)^{\nu-1}ds\varphi(r)dr \notag\\
&= \int_0^t (t-r)^{(n+2)\nu-1} \int_0^1  (1-u)^{(n+1)\nu-1} u^{\nu-1}du  \varphi(r)dr \notag\\
&= B(\nu,(n+1)\nu)\int_0^t (t-r)^{(n+2)\nu-1}  \varphi(r)dr \notag\\
&= \frac{\Gamma(n+1)\nu)\Gamma(\nu)}{\Gamma(n+2)\nu)}\int_0^t (t-r)^{(n+2)\nu-1}  \varphi(r)dr. \label{E4}
\end{align}
Thanks to \eqref{E2}, \eqref{E3} and \eqref{E4}, the estimate  \eqref{E1}   holds true for $n+1$.

Since $\varphi$ is bounded on  $[a,b]$, the second term in the right-hand side of  \eqref{E1} is estimated by:
\begin{align*}
&a^{-\mu(n+1)}  \frac{\Gamma(\nu)^{n+1}}{\Gamma((n+1)\nu)} \int_a^t (t-s)^{(n+1)\nu-1} \varphi(s)ds \\
&\leq   \frac{a^{-\mu(n+1)}\Gamma(\nu)^{n+1}(t-a)^{(n+1)\nu}\sup_{s\in[a,t]} \varphi(s)}{(n+1)\nu)\Gamma((n+1)\nu)}.
\end{align*}
Due to  the Stirling's formula, it is known that 
$$\Gamma (x+1) \sim \sqrt{2\pi x} \Big(\frac{x}{e}\Big)^x \hspace{2cm} \text{ as } x\to \infty.$$
This  term therefore converges to zero as $n\to \infty$. As a consequence,
\begin{align*}
\varphi(t)\leq &f(t) \sum_{k=0}^\infty   \frac{[a^{-\mu}t^{\nu}\Gamma(\nu)]^k}{\Gamma(1+k\nu)}, \hspace{2cm} a \leq t \leq b.
\end{align*}
It is known that (e.g., \cite[Lemma 1.2]{yagi})
\begin{align*}
&\sum_{k=0}^\infty   \frac{[a^{-\mu}t^{\nu}\Gamma(\nu)]^k}{\Gamma(1+k\nu)}  \\
& \leq \frac{2}{\min_{0<s<\infty} \Gamma(s) \nu} (1+t[a^{-\mu}\Gamma(\nu)]^{\frac{1}{\nu}}) e^{t[a^{-\mu}\Gamma(\nu)]^{\frac{1}{\nu}}+1}, \hspace{1cm} 0\leq t<\infty.
\end{align*}
 Thus, the lemma has been proved.
\end{proof}

\begin{proof}[Proof for Theorem \ref{Thm6}]
This theorem is proved by using analogous arguments as in  the proof for Theorem \ref{Thm5}. 
We  use a universal constant $C_{\mathcal B_1, \mathcal B_2, B_A}$, which depends only on the exponents and $\mathcal B_1, \mathcal B_2$ and $ B_A.$

First, let us give an estimate for 
$$t^{2\eta}\mathbb E [ \|A^\eta[X(t)-\bar X(t)]\|^2+\|A^{\beta} [X(t)-\bar X(t)]\|^2].$$
For $0 \leq \theta <\frac{1}{2}$ and $0<t\leq T_{\mathcal B_1, \mathcal B_2, B_A}$,  \eqref{P5}, \eqref{P6} and {\rm (H3)} give 
\begin{align*}
&t^\theta \|A^\theta[X(t)-\bar X(t)]\|\\
= & \Big|\Big|t^\theta A^\theta S(t)(\xi-\bar \xi)+\int_0^t t^\theta A^\theta S(t-s)[F_1(X(s))-F_1(\bar X(s))]ds\notag\\
&+\int_0^t t^\theta A^\theta S(t-s) [F_2(s)-\bar F_2(s)]ds\notag\\
&+\int_0^t t^\theta A^\theta S(t-s) [G(s)-\bar G(s)]dW(s)\Big|\Big| \notag\\
\leq & \iota_\theta \|\xi-\bar \xi\|+\iota_\theta c_{F_1}\int_0^t  t^\theta (t-s)^{-\theta}  \|A^\eta[X(s)-\bar X(s)]\| ds\notag\\
&+\iota_\theta \|F_2-\bar F_2\|_{\mathcal F^{\beta,\sigma}} \int_0^t t^\theta (t-s)^{-\theta}s^{\beta-1}ds\notag\\
&+\Big|\Big|\int_0^t t^\theta A^\theta S(t-s) [G(s)-\bar G(s)]dW(s)\Big|\Big| \notag\\
= & \iota_\theta \|\xi-\bar \xi\|+\iota_\theta \|F_2-\bar F_2\|_{\mathcal F^{\beta,\sigma}}  B(\beta,1-\theta)t^\beta\notag\\
& +\iota_\theta c_{F_1}\int_0^t  t^\theta (t-s)^{-\theta}  \|A^\eta[X(s)-\bar X(s)]\| ds\notag\\
&+\Big|\Big|\int_0^t t^\theta A^\theta S(t-s) [G(s)-\bar G(s)]dW(s)\Big|\Big|. \notag
\end{align*}
Thus,
\begin{align}
&\mathbb E\|t^\theta A^\theta[X(t)-\bar X(t)]\|^2 \notag\\
\leq & 4\iota_\theta^2 \mathbb E \|\xi-\bar \xi\|^2+4\iota_\theta^2 \|F_2-\bar F_2\|_{\mathcal F^{\beta,\sigma}}^2   B(\beta,1-\theta)^2 t^{2\beta}\notag\\
& +4\iota_\theta^2 c_{F_1}^2 t^{2\theta} \mathbb E \Big[ \int_0^t  (t-s)^{-\theta}  \|A^\eta[X(s)-\bar X(s)]\| ds\Big]^2\notag\\
&+4\mathbb E \Big|\Big|\int_0^t t^\theta A^\theta S(t-s) [G(s)-\bar G(s)]dW(s)\Big|\Big|^2 \notag\\
\leq & 4\iota_\theta^2 \mathbb E \|\xi-\bar \xi\|^2+4\iota_\theta^2 \|F_2-\bar F_2\|_{\mathcal F^{\beta,\sigma}}^2   B(\beta,1-\theta)^2 t^{2\beta}\notag\\
& +4\iota_\theta^2 c_{F_1}^2 t^{2\theta} \mathbb E \Big[ \int_0^t  (t-s)^{-\theta}  \|A^\eta[X(s)-\bar X(s)]\| ds\Big]^2\notag\\
&+4 \int_0^t \|t^\theta A^\theta S(t-s)\|^2 \|G(s)-\bar G(s)\|_{L_2(U;H)}^2ds \notag\\
\leq & 4\iota_\theta^2 \mathbb E \|\xi-\bar \xi\|^2+4\iota_\theta^2 \|F_2-\bar F_2\|_{\mathcal F^{\beta,\sigma}}^2   B(\beta,1-\theta)^2 t^{2\beta} \notag\\
&+4\iota_\theta^2 c_{F_1}^2 t^{2\theta+1} \int_0^t   (t-s)^{-2\theta}  \mathbb E \|A^\eta[X(s)-\bar X(s)]\|^2 ds\notag\\
&+4\iota_\theta^2 \|G-\bar G\|_{B([0,T];L_2(U;H))}^2\int_0^t t^{2\theta} (t-s)^{-2\theta} ds \notag\\
\leq  & 4\iota_\theta^2 \mathbb E \|\xi-\bar \xi\|^2+4\iota_\theta^2   B(\beta,1-\theta)^2 t^{2\beta} \|F_2-\bar F_2\|_{\mathcal F^{\beta,\sigma}}^2\label{P39}\\
&+\frac{4\iota_\theta^2 t }{1-2\theta}\|G-\bar G\|_{B([0,T];L_2(U;H))}^2\notag\\
&+4\iota_\theta^2 c_{F_1}^2 t^{2\theta+1}  \int_0^t   (t-s)^{-2\theta}  \mathbb E \|A^\eta[X(s)-\bar X(s)]\|^2 ds.\notag
\end{align}

Applying these estimates  with $\theta=\beta$ and $\theta=\eta$, we have
\begin{align*}
&\mathbb E\| A^\beta[X(t)-\bar X(t)]\|^2 \\
\leq  & 4\iota_\beta^2 \mathbb E \|\xi-\bar \xi\|^2t^{-2\beta}+4\iota_\beta^2   B(\beta,1-\beta)^2  \|F_2-\bar F_2\|_{\mathcal F^{\beta,\sigma}}^2\notag\\
&+\frac{4\iota_\beta^2 t^{1-2\beta} }{1-2\beta} \|G-\bar G\|_{B([0,T];L_2(U;H))}^2\notag\\
&+4\iota_\beta^2 c_{F_1}^2 t  \int_0^t   (t-s)^{-2\beta}  \mathbb E \|A^\eta[X(s)-\bar X(s)]\|^2 ds,  \hspace{1cm} 0<t\leq T_{\mathcal B_1, \mathcal B_2, B_A},\notag
\end{align*}
and
\begin{align*}
&t^{2\eta} \mathbb E\|A^\eta[X(t)-\bar X(t)]\|^2\\
\leq  & 4\iota_\eta^2 \mathbb E \|\xi-\bar \xi\|^2+4\iota_\eta^2   B(\beta,1-\eta)^2 t^{2\beta} \|F_2-\bar F_2\|_{\mathcal F^{\beta,\sigma}}^2\notag\\
&+\frac{4\iota_\eta^2 t }{1-2\eta} \|G-\bar G\|_{B([0,T];L_2(U;H))}^2\notag\\
&+4\iota_\eta^2 c_{F_1}^2 t^{2\eta+1}  \int_0^t   (t-s)^{-2\eta}  \mathbb E \|A^\eta[X(s)-\bar X(s)]\|^2 ds,  \hspace{0.7cm} 0<t\leq T_{\mathcal B_1, \mathcal B_2, B_A}.\notag
\end{align*}
By putting 
$$q(t)=t^{2\eta}\mathbb E [ \|A^\eta[X(t)-\bar X(t)]\|^2+\|A^{\beta} [X(t)-\bar X(t)]\|^2], $$
we then obtain an integral inequality 
\begin{align}
q(t)
\leq  & 4[\iota_\beta^2 t^{2(\eta-\beta)}+\iota_\eta^2 ] \mathbb E \|\xi-\bar \xi\|^2\notag\\
&+4[\iota_\beta^2   B(\beta,1-\beta)^2t^{2\eta}+\iota_\eta^2   B(\beta,1-\eta)^2 t^{2\beta} ]  \|F_2-\bar F_2\|_{\mathcal F^{\beta,\sigma}}^2\notag\\
&+4\Big[\frac{\iota_\beta^2 t^{1+2\eta-2\beta} }{1-2\beta}+\frac{\iota_\eta^2 t }{1-2\eta}\Big] \|G-\bar G\|_{B([0,T];L_2(U;H))}^2 \notag\\
&+4 c_{F_1}^2 t^{2\eta+1}\int_0^t [ \iota_\beta^2 (t-s)^{-2\beta}+\iota_\eta^2  (t-s)^{-2\eta}]  s^{-2\eta}q(s) ds\notag\\
\leq  & C_{\mathcal B_1, \mathcal B_2, B_A}[\mathbb E \|\xi-\bar \xi\|^2+ t^{2\beta}   \|F_2-\bar F_2\|_{\mathcal F^{\beta,\sigma}}^2 \label{P40}
\\
&+ t \|G-\bar G\|_{B([0,T];L_2(U;H))}^2]\notag \\
&+4 c_{F_1}^2 t^{2\eta+1}\int_0^t [ \iota_\beta^2 (t-s)^{-2\beta}+\iota_\eta^2  (t-s)^{-2\eta}]  \notag \\
&\hspace{2cm}  \times  s^{-2\eta}q(s) ds,  \hspace{1cm} 0<t\leq T_{\mathcal B_1, \mathcal B_2, B_A}.  \notag
\end{align}

We  solve the integral inequality  \eqref{P40} as follows. Let $\epsilon>0$  denote a small parameter. For $0 \leq t \leq \epsilon,$ 
\begin{align*}
&q(t)\\
\leq   & C_{\mathcal B_1, \mathcal B_2, B_A}[\mathbb E \|\xi-\bar \xi\|^2+ t^{2\beta}   \|F_2-\bar F_2\|_{\mathcal F^{\beta,\sigma}}^2 + t \|G-\bar G\|_{B([0,T];L_2(U;H))}^2]
\\
&+4 c_{F_1}^2 t^{2\eta+1}\int_0^t [ \iota_\beta^2 (t-s)^{-2\beta}+\iota_\eta^2  (t-s)^{-2\eta}]  s^{-2\eta} ds   \sup_{s\in [0,\epsilon]} q(s)\\
=& C_{\mathcal B_1, \mathcal B_2, B_A}[\mathbb E \|\xi-\bar \xi\|^2+ t^{2\beta}   \|F_2-\bar F_2\|_{\mathcal F^{\beta,\sigma}}^2 + t \|G-\bar G\|_{B([0,T];L_2(U;H))}^2]
\\
&+4 c_{F_1}^2  [ \iota_\beta^2 B(1-2\eta,1-2\beta)t^{2(1-\beta)}+\iota_\eta^2  B(1-2\eta,1-2\eta) t^{2(1-\eta)}]     \sup_{s\in [0,\epsilon]} q(s)\\
\leq &C_{\mathcal B_1, \mathcal B_2, B_A}[\mathbb E \|\xi-\bar \xi\|^2+ \epsilon^{2\beta}   \|F_2-\bar F_2\|_{\mathcal F^{\beta,\sigma}}^2 + \epsilon \|G-\bar G\|_{B([0,T];L_2(U;H))}^2]
\\
&+4 c_{F_1}^2  [ \iota_\beta^2 B(1-2\eta,1-2\beta)\epsilon^{2(1-\beta)}+\iota_\eta^2  B(1-2\eta,1-2\eta) \epsilon^{2(1-\eta)}]     \sup_{s\in [0,\epsilon]} q(s).
\end{align*}
Taking the supremum on both the hand sides of the above inequality, we observe that
\begin{align*}
&\{1-4 c_{F_1}^2  [ \iota_\beta^2 B(1-2\eta,1-2\beta)\epsilon^{2(1-\beta)}+\iota_\eta^2  B(1-2\eta,1-2\eta) \epsilon^{2(1-\eta)}] \}    \sup_{s\in [0,\epsilon]} q(s)\\
&\leq C_{\mathcal B_1, \mathcal B_2, B_A}[\mathbb E \|\xi-\bar \xi\|^2+ \epsilon^{2\beta}   \|F_2-\bar F_2\|_{\mathcal F^{\beta,\sigma}}^2 + \epsilon \|G-\bar G\|_{B([0,T];L_2(U;H))}^2].
\end{align*}
If $\epsilon$ is taken sufficiently small so that 
\begin{equation} \label{P41}
1-4 c_{F_1}^2  [ \iota_\beta^2 B(1-2\eta,1-2\beta)\epsilon^{2(1-\beta)}+\iota_\eta^2  B(1-2\eta,1-2\eta) \epsilon^{2(1-\eta)}] \geq \frac{1}{2},
\end{equation}
 then
\begin{align}  
\sup_{s\in [0,\epsilon]} q(s)  \leq &C_{\mathcal B_1, \mathcal B_2, B_A}[\mathbb E \|\xi-\bar \xi\|^2+ \epsilon^{2\beta}   \|F_2-\bar F_2\|_{\mathcal F^{\beta,\sigma}}^2 \label{P42}\\
&+ \epsilon \|G-\bar G\|_{B([0,T];L_2(U;H))}^2]. \notag
\end{align}
As a consequence,
\begin{align}  
 q(\epsilon) \leq &C_{\mathcal B_1, \mathcal B_2, B_A}[\mathbb E \|\xi-\bar \xi\|^2+ \epsilon^{2\beta}   \|F_2-\bar F_2\|_{\mathcal F^{\beta,\sigma}}^2 \label{P43} \\
& + \epsilon \|G-\bar G\|_{B([0,T];L_2(U;H))}^2] \notag
\end{align}
for any  $\epsilon$ satisfying \eqref{P41}.

In the meantime, for $\epsilon < t \leq  T_{\mathcal B_1, \mathcal B_2, B_A}, $ 
\begin{align*}
q(t)
\leq   & C_{\mathcal B_1, \mathcal B_2, B_A}[\mathbb E \|\xi-\bar \xi\|^2+ t^{2\beta}   \|F_2-\bar F_2\|_{\mathcal F^{\beta,\sigma}}^2 + t \|G-\bar G\|_{B([0,T];L_2(U;H))}^2]
\\
&+4 c_{F_1}^2 t^{2\eta+1}\int_0^\epsilon [ \iota_\beta^2 (t-s)^{-2\beta}+\iota_\eta^2  (t-s)^{-2\eta}]  s^{-2\eta} ds \sup_{s\in [0,\epsilon]} q(s) \\
&+4 c_{F_1}^2 t^{2\eta+1}\int_\epsilon^t [ \iota_\beta^2 (t-s)^{-2\beta}+\iota_\eta^2  (t-s)^{-2\eta}]  s^{-2\eta}q(s) ds \\
\leq   & C_{\mathcal B_1, \mathcal B_2, B_A}[\mathbb E \|\xi-\bar \xi\|^2+ t^{2\beta}   \|F_2-\bar F_2\|_{\mathcal F^{\beta,\sigma}}^2 + t \|G-\bar G\|_{B([0,T];L_2(U;H))}^2]
\\
&+4 c_{F_1}^2  [ \iota_\beta^2 B(1-2\eta,1-2\beta)\epsilon^{2(1-\beta)}+\iota_\eta^2  B(1-2\eta,1-2\eta) \epsilon^{2(1-\eta)}]    \\
&\times   \sup_{s\in [0,\epsilon]} q(s) \\
&+4 c_{F_1}^2 t^{2\eta+1}\int_\epsilon^t [ \iota_\beta^2 (t-s)^{2(\eta-\beta)}+\iota_\eta^2  ] (t-s)^{-2\eta} \epsilon^{-2\eta}q(s) ds \\
\leq   & C_{\mathcal B_1, \mathcal B_2, B_A}[\mathbb E \|\xi-\bar \xi\|^2+ t^{2\beta}   \|F_2-\bar F_2\|_{\mathcal F^{\beta,\sigma}}^2 + t \|G-\bar G\|_{B([0,T];L_2(U;H))}^2]
\\
&+4 c_{F_1}^2  [ \iota_\beta^2 B(1-2\eta,1-2\beta)\epsilon^{2(1-\beta)}+\iota_\eta^2  B(1-2\eta,1-2\eta) \epsilon^{2(1-\eta)}]   \\
&\times  \sup_{s\in [0,\epsilon]} q(s) +4 c_{F_1}^2 \epsilon^{-2\eta} T^{2\eta+1} [ \iota_\beta^2 T^{2(\eta-\beta)}+\iota_\eta^2  ] \int_\epsilon^t  (t-s)^{-2\eta} q(s) ds.
\end{align*}
Lemma \ref{Thm3} then provides that 
\begin{align*}
q(t) &\\
\leq & C_{\mathcal B_1, \mathcal B_2, B_A}[\mathbb E \|\xi-\bar \xi\|^2+ t^{2\beta}   \|F_2-\bar F_2\|_{\mathcal F^{\beta,\sigma}}^2 + t \|G-\bar G\|_{B([0,T];L_2(U;H))}^2]
\\
&+4 c_{F_1}^2  [ \iota_\beta^2 B(1-2\eta,1-2\beta)\epsilon^{2(1-\beta)}+\iota_\eta^2  B(1-2\eta,1-2\eta) \epsilon^{2(1-\eta)}]   \\
&\times  \sup_{s\in [0,\epsilon]} q(s)\\
\leq & C_{\mathcal B_1, \mathcal B_2, B_A}[\mathbb E \|\xi-\bar \xi\|^2+ t^{2\beta}   \|F_2-\bar F_2\|_{\mathcal F^{\beta,\sigma}}^2 + t \|G-\bar G\|_{B([0,T];L_2(U;H))}^2]
\\
&+4 c_{F_1}^2  [ \iota_\beta^2 B(1-2\eta,1-2\beta)T^{2(1-\beta)} +\iota_\eta^2  B(1-2\eta,1-2\eta) T^{2(1-\eta)}]  \\
&  \times  \sup_{s\in [0,\epsilon]} q(s), \hspace{2cm} \epsilon < t \leq T_{\mathcal B_1, \mathcal B_2, B_A}.
\end{align*} 
Thanks to \eqref{P42},
\begin{align*}
q(t) &\\
\leq & C_{\mathcal B_1, \mathcal B_2, B_A}[\mathbb E \|\xi-\bar \xi\|^2+ t^{2\beta}   \|F_2-\bar F_2\|_{\mathcal F^{\beta,\sigma}}^2 + t \|G-\bar G\|_{B([0,T];L_2(U;H))}^2]
\\
&+C_{\mathcal B_1, \mathcal B_2, B_A}[\mathbb E \|\xi-\bar \xi\|^2+ \epsilon^{2\beta}   \|F_2-\bar F_2\|_{\mathcal F^{\beta,\sigma}}^2 \\
&+ \epsilon \|G-\bar G\|_{B([0,T];L_2(U;H))}^2], \hspace{2cm} \epsilon < t \leq T_{\mathcal B_1, \mathcal B_2, B_A}.
\end{align*}
 Hence,
\begin{align}
q(t) \leq & C_{\mathcal B_1, \mathcal B_2, B_A}  [\mathbb E \|\xi-\bar \xi\|^2+ t^{2\beta}   \|F_2-\bar F_2\|_{\mathcal F^{\beta,\sigma}}^2 \label{P44}  \\
&\hspace{1cm}+ t \|G-\bar G\|_{B([0,T];L_2(U;H))}^2], \hspace{1.5cm}\epsilon < t \leq T_{\mathcal B_1, \mathcal B_2, B_A}. \notag
\end{align}

Combining \eqref{P43} and \eqref{P44}, we conclude that  
\begin{align}
t^{2\eta} &  \mathbb E [ \|A^\eta[X(s)-\bar X(s)]\|^2+\|A^{\beta} [X(s)-\bar X(s)]\|^2]
\notag\\
=&q(t) \notag\\
\leq & C_{\mathcal B_1, \mathcal B_2, B_A}[\mathbb E \|\xi-\bar \xi\|^2+ t^{2\beta}   \|F_2-\bar F_2\|_{\mathcal F^{\beta,\sigma}}^2 \label{P45} 
\notag \\
&+ t \|G-\bar G\|_{B([0,T];L_2(U;H))}^2],  \hspace{1cm}  0 <t \leq T_{\mathcal B_1, \mathcal B_2, B_A}.
\end{align}

Second, let us  give an estimate for $\mathbb E\|X(t)-\bar X(t)\|^2$. Taking $\theta=0$ in \eqref{P39}, we have
\begin{align*}
&\mathbb E\|X(t)-\bar X(t)\|^2  \\
\leq  &  4\iota_0 \mathbb E \|\xi-\bar \xi\|^2+ 4\iota_0  B(\beta,1)^2 t^{2\beta} \|F_2-\bar F_2\|_{\mathcal F^{\beta,\sigma}}^2\notag\\
&+4 \iota_0 t \|G-\bar G\|_{B([0,T];L_2(U;H))}^2 +4\iota_0^2  c_{F_1}^2 t\int_0^t   s^{-2\eta}q(s)ds.\notag
\end{align*}
The term $t\int_0^t   s^{-2\eta}q(s)ds$ is estimated by  using \eqref{P45}:
 \begin{align*}
 & t\int_0^t   s^{-2\eta}q(s)ds\\
 \leq &C_{\mathcal B_1, \mathcal B_2, B_A} t\int_0^t   s^{-2\eta}[\mathbb E \|\xi-\bar \xi\|^2+ s^{2\beta}   \|F_2-\bar F_2\|_{\mathcal F^{\beta,\sigma}}^2 \\
&\hspace{2cm} + s \|G-\bar G\|_{B([0,T];L_2(U;H))}^2]ds\notag\\
  \leq &\frac{C_{\mathcal B_1, \mathcal B_2, B_A} t^{2(1-\eta)}\mathbb E \|\xi-\bar \xi\|^2}{1-2\eta}+\frac{C_{\mathcal B_1, \mathcal B_2, B_A} t^{2(1+\beta-\eta)}
\|F_2-\bar F_2\|_{\mathcal F^{\beta,\sigma}}^2}{1+2\beta-2\eta}\\
&+ \frac{C_{\mathcal B_1, \mathcal B_2, B_A} t^{3-2\eta}\|G-\bar G\|_{B([0,T];L_2(U;H))}^2}{2(1-\eta)}.
  \end{align*}
Therefore, 
 \begin{align}
 \mathbb E& \|X(t)-\bar X(t)\|^2  \leq  C_{\mathcal B_1, \mathcal B_2, B_A}[ \mathbb E \|\xi-\bar \xi\|^2+ t^{2\beta}   \|F_2-\bar F_2\|_{\mathcal F^{\beta,\sigma}}^2 \label{P46}\\
 &+ t \|G-\bar G\|_{B([0,T];L_2(U;H))}^2], \hspace{1cm}  0 <t \leq T_{\mathcal B_1, \mathcal B_2, B_A}. \notag
\end{align}
Thanks to  \eqref{P45} and \eqref{P46}, the estimate  \eqref{P37} has been verified.

Finally, let us  prove the estimate  \eqref{P38}. By substituting the estimate
$$\|A^\theta S(t) (\xi-\bar \xi)\|\leq \iota_{\theta-\beta} t^{\beta-\theta} \|A^\beta  (\xi-\bar \xi)\|$$
with $\theta=\beta$ and $\theta=\eta$ for 
$\|A^\theta S(t) (\xi-\bar \xi)\|\leq \iota_\theta t^{-\theta} \|\xi-\bar \xi\|,$ we obtain a similar result to \eqref{P39}:
\begin{align*}
&\mathbb E\|t^\theta A^\theta[X(t)-\bar X(t)]\|^2\notag\\
\leq  & 4 \iota_{\theta-\beta}^2 t^{2(\beta-\theta)} \mathbb E \|A^\beta  (\xi-\bar \xi)\|^2
+4\iota_\theta^2   B(\beta,1-\theta)^2 t^{2\beta} \|F_2-\bar F_2\|_{\mathcal F^{\beta,\sigma}}^2\notag\\
&+\frac{4\iota_\theta^2 t}{1-2\theta} \|G-\bar G\|_{B([0,T];L_2(U;H))}^2\notag\\
&+4\iota_\theta^2 c_{F_1}^2 t^{2\theta+1}  \int_0^t   (t-s)^{-2\theta}  \mathbb E \|A^\eta[X(s)-\bar X(s)]\|^2] ds.\notag
\end{align*}
Using the same arguments as for \eqref{P45}, we conclude that  \eqref{P38} holds true. 
 This completes the proof of the theorem.
\end{proof}

\section {Applications to stochastic PDEs}\label{section4}
We present some applications to stochastic  PDEs. These equations are considered under both Neumann and Dirichlet type boundary conditions.

\subsection {Example 1} 
Consider the stochastic PDE
\begin{equation} \label{P47}
\begin{cases}
du(x,t)=\{ [a(x) u'(x)]'  +\frac{u}{1+u^2}+ t^{\beta-1}f(t) \varphi_1(x)\}dt +g(t) \varphi_2(x)dw_t \\
  \hspace{6.25cm} \text { in } (0,1) \times (0,T),\\
  u'(0)=u'(1)=0,   \\
u(x,0)=u_0(x) \hspace{4cm} \text { in }  (0,1).
\end{cases}
\end{equation}
Here,  $w$ is a real-valued standard Wiener process. The functions  $\varphi_i (i=1,2)$ are real-valued and square integrable  on  $(0,1)$. Meanwhile, $f$ is a real-valued $\sigma$-H\"{o}lder continuous function on $[0,T]$ with $0<\sigma<\beta<1$, $g$ is a real-valued bounded function on $[0,T]$, and 
   $a$ is a real-value function on $(0,1)$   satisfying the condition
$$
a \in \mathcal C^1([0,1]) \quad \text{ and }  \quad a(x) \geq a_0, \hspace{2cm}   0<x<1
$$
with some constant $a_0>0.$ 

We  handle the equation \eqref{P47} in the Hilbert space  $L_2((0,1))$.  Let $A$ be the  realization of the differential operator 
$$ -\frac{d }{dx} [a(x)\frac{d }{dx}] +1$$
 in $L_2((0,1))$ under the Neumann type boundary conditions:
$$u'(0)=u'(1)=0.$$
 According to \cite[Theorem 2.12]{yagi},  $A$ is a sectorial operator of $L_2((0,1))$   with domain
$$\mathcal D(A)=\{u\in H^2((0,1)); u'(0)=u'(1)=0\}.$$

Using $A$,  \eqref{P47} is formulated as an abstract problem of the form \eqref{P1}.
 Here, the nonlinear operator $F_1$ is given by
$$F_1(u)=u+\frac{u}{1+u^2},$$
and the functions $F_2\colon (0,T] \to L_2((0,1))$ and $G\colon [0,T] \to L_2(\mathbb R;L_2((0,1)))$ are defined by
$$F_2(t)=t^{\beta-1} f(t) \varphi(x), \quad G(t)=g(t)\varphi(x).$$

Obviously,  
$$ G\in B([0,T];L_2(\mathbb R;L_2((0,1))))$$
 and 
$$F_2 \in \mathcal F^{\beta,\sigma} ((0,T];L_2((0,1))) \hspace{1cm} (\text{see Remark } \ref{remark1}). $$

Let  fix $\eta$ such that 
\begin{equation*}
\begin{cases}
0<\eta<\frac{1}{2}, \\
 \max\{0, 2\eta-\frac{1}{2}\}<\beta<\eta.
\end{cases}
\end{equation*}
For  $u,v\in \mathcal D(A^\eta), $ we have
\begin{align*}
&\|F_1(u)-F_1(u)\|_{L_2((0,1))}\\
&=\|u-v+\frac{u}{1+u^2} -\frac{v}{1+v^2}\|_{L_2((0,1))}\\
&\leq \|u-v\|_{L_2((0,1))} + \Big|\Big|\frac{(uv-1)(u-v)}{(1+u^2)(1+v^2)} \Big|\Big|_{L_2((0,1))}\\
&\leq C_1\|u-v\|_{L_2((0,1))}  \leq C_2\|A^\eta(u-v)\|_{L_2((0,1))}
\end{align*}
with some constants $C_1, C_2>0$.

All the structural  assumptions are therefore satisfied in $L_2((0,1))$. By using Theorems  \ref{Thm5}, \ref{Thm7} and \ref{Thm6}, we have the following results.

{\bf Claim 1} (existence of unique solutions). 
Let 
  $u_0 \in \mathcal D(A^\beta) $ a.s. with $\mathbb E \|A^\beta u_0\|^2$ $<\infty.$ Then, 
  \eqref{P47} possesses a unique local mild solution $u$ in the function space:
\begin{equation*}
u\in  \mathcal C([0,T_{loc}];\mathcal D(A^\beta)), \quad A^\eta u\in \mathcal C^\gamma ((0,T_{loc}];L_2((0,1)))  \hspace{1cm} \text{a.s.}
\end{equation*}
for any $ 0<\gamma<\frac{1+2\beta}{4}-\eta.$ 
 Furthermore, $u$ satisfies the estimate:
\begin{equation*}
\mathbb E \|A^\beta u(t)\|^2 +t^{2(\eta-\beta)} \mathbb E \|A^\eta u(t)\|^2 \leq C_{F_2,G,u_0},    \hspace{2cm} 0\leq t\leq T_{loc}.
\end{equation*}
Here, $T_{loc}$ and  $C_{F_2,G,u_0}$ are some constants depending  on the exponents and   $ \mathbb E \|A^\beta u_0\|^2,$  $\|F_2\|_{\mathcal F^{\beta,\sigma}}^2$, $ \|G\|_{B([0,T];L_2(\mathbb R;L_2((0,1))))}^2.$  

{\bf Claim 2} (differentiability of expectation). 
Let $\sigma+\eta\leq \frac{1}{2}.$ Put $z(t)=\mathbb E u(t)$, then
\begin{equation*} 
\begin{cases}
z\in  \mathcal C((0,T_{loc}];\mathcal D(A))\cap \mathcal C([0,T_{loc}];\mathcal D(A^\beta))\cap \mathcal C^1((0,T_{loc}];L_2((0,1))),\\
\frac{dz}{dt}, Az \in \mathcal F^{\beta,\sigma}((0,T_{loc}];L_2((0,1))).
\end{cases}
\end{equation*}
 Furthermore, $z$ satisfies the estimate
\begin{equation*}
\|A^\beta z\|_{\mathcal C} + \Big|\Big|\frac{dz}{dt} \Big|\Big|_{\mathcal F^{\beta,\sigma}}+ \|Az\|_{\mathcal F^{\beta,\sigma}}\leq C_{F_2,u_0},    \hspace{2cm} 0\leq t\leq T_{loc}
\end{equation*}
with some constant  $C_{F_2,u_0}$ depending  on    $ \mathbb E \|A^\beta u_0\|^2$,  $\|F_2\|_{\mathcal F^{\beta,\sigma}}^2$ and  $T_{loc}.$

{\bf Claim 3} (continuous dependence on initial data). Let $\mathcal B_1$ and $\mathcal B_2$ be bounded balls:
$$\mathcal B_1=\{f\in \mathcal F^{\beta,\sigma}((0,T];L_2((0,1)));  \|f\|_{\mathcal F^{\beta,\sigma}}\leq R_1\}, \hspace{1cm} 0<R_1<\infty,  $$
\begin{align*}
\mathcal B_2=\{& g\in B([0,T];L_2(U;L_2((0,1)))); \\
& \|G\|_{B([0,T];L_2(\mathbb R;L_2((0,1))))}\leq R_2\}, \hspace{1cm} 0<R_2<\infty,
\end{align*}
of the spaces $\mathcal F^{\beta,\sigma}((0,T];L_2((0,1)))$ and $B([0,T];L_2(U;H))$, respectively.
Let $B_A$ be a set of random variable:
\begin{equation*}  
B_A=\{\zeta;  \zeta\in \mathcal D(A^\beta) \, \text{   a.s. and    }  \,   \mathbb E \|A^\beta \zeta\|^2\leq R_3^2\}, \quad 0<R_3<\infty.
\end{equation*}

Let $u$ and $\bar u$ be the solutions of  \eqref{P47} for the data $(F_2,G,u_0)$ and $(\bar F_2,\bar G,\bar u_0)$ in $\mathcal B_1\times \mathcal B_2\times B_A$, respectively. Then, there exist  constants $T_{\mathcal B_1, \mathcal B_2, B_A}$ and $C_{\mathcal B_1, \mathcal B_2, B_A}$ depending only on $\mathcal B_1, \mathcal B_2$ and $ B_A$ such that
\begin{align*}
&t^{2\eta}\mathbb E  \|A^\eta[u(t)-\bar u(t)]\|^2+t^{2\eta}\mathbb E\|A^{\beta} [u(t)-\bar u(t)]\|^2+ \mathbb E\|u(t)-\bar u(t)\|^2 \\
\leq &C_{\mathcal B_1, \mathcal B_2, B_A}[\mathbb E \|u_0-\bar u_0\|^2+ t^{2\beta}   \|F_2-\bar F_2\|_{\mathcal F^{\beta,\sigma}}^2 \notag\\
&\hspace{2cm}+ t \|G-\bar G\|_{B([0,T];L_2(\mathbb R;L_2((0,1))))}^2], \hspace{1cm} 0<t\leq T_{\mathcal B_1, \mathcal B_2, B_A}\notag
\end{align*}
and 
\begin{align*}
&t^{2(\eta-\beta)}[\mathbb E  \|A^\eta[u(t)-\bar u(t)]\|^2+\mathbb E\|A^{\beta} [u(t)-\bar u(t)]\|^2] \\
& \leq C_{\mathcal B_1, \mathcal B_2, B_A}[\mathbb E \|A^\beta(u_0-\bar u_0)\|^2+    \|F_2-\bar F_2\|_{\mathcal F^{\beta,\sigma}}^2\notag\\
&\hspace{2cm}+ \|G-\bar G\|_{B([0,T];L_2(\mathbb R;L_2((0,1))))}^2], \hspace{1cm} 0<t\leq T_{\mathcal B_1, \mathcal B_2, B_A}. \notag
\end{align*}

\subsection {Example 2}
Let us consider the initial value problem
\begin{equation} \label{P48}
\begin{cases}
\begin{aligned}
du(x,t)=&\Big[\sum_{i,j=1}^n \frac{\partial }{\partial x_j} \Big[a_{ij}(x)\frac{\partial }{\partial x_i}u\Big] +f_1(u)+t^{\beta-1}f_2(t) \varphi(x)\Big]dt  \\
  &+g(t) dW(t)  \hspace{3cm}   \text { in } \mathbb R^n \times (0,T),\\
 u(x,0)=&u_0(x)    \hspace{4.3cm}  \text { in }  \mathbb R^n.
\end{aligned}
\end{cases}
\end{equation}
Here,  $W$ is a cylindrical  Wiener process on some separable Hilbert space $U$; $\varphi $  is a   function in $H^{-1}(\mathbb R^n)$; $f_1$ is a real-valued function on $\mathbb R$; $f_2$ is a $\sigma$-H\"{o}lder continuous function on $[0,T]$ with $0<\sigma<\beta<1$; $g$ is a   $L_2(U;H^{-1}(\mathbb R^n))$-valued bounded function on $[0,T]$;  
 $a_{ij} (x), 1\leq i,j\leq n,$  are real-valued functions in $\mathbb R^n$ satisfying the conditions:
\begin{itemize}
\item $\sum_{i,j=1}^n a_{ij}(x) z_i z_j \geq a_0 \|z\|^2,$  \hspace{0.5cm} $ z=(z_1,\dots,z_n)\in \mathbb R^n,$  \quad a.e. $ x\in \mathbb R^n$ with some constant $a_0>0$
\item $a_{ij}\in L_\infty (\mathbb R^n) $  \hspace{1cm} a.e. $x\in \mathbb R^n $
\end{itemize}

We  handle \eqref{P48} in the Hilbert space   $H^{-1}(\mathbb R^n).$
 Let $A$ be the realization of the differential operator 
$$-\sum_{i,j=1}^n \frac{\partial }{\partial x_j} [a_{ij}(x)\frac{\partial }{\partial x_i}] +1$$
 in $H^{-1}(\mathbb R^n).$ 
Thanks to \cite[Theorem 2.2]{yagi},  $A$ is a sectorial operator of $H^{-1}(\mathbb R^n)$   with domain
$\mathcal D(A)=H^1(\mathbb R^n)$. As a consequence, $(-A)$ generates an analytical semigroup on $H^{-1}(\mathbb R^n)$. 
Using $A$,  the equation \eqref{P48} is formulated as an abstract problem of the form \eqref{P1}, where  $F_1, F_2$ and $G$ are defined as follows.
The nonlinear operator $F_1$ is given by
$$F(u)=u+f_1(u).$$ 
Here, we assume that this function is defined on the domain of $A^\eta$ and satisfies the condition:
$$\|F_1(u)-F_1(v)\|_{H^{-1}(\mathbb R^n)} \leq c \|A^\eta (u-v)\|_{H^{-1}(\mathbb R^n)},  \hspace{1cm}  u,v \in  \mathcal D(A^\eta)$$
with $c>0$, $0<\eta<\frac{1}{2} $ and $
 \max\{0, 2\eta-\frac{1}{2}\}<\beta<\eta.$
The term $F_2\colon (0,T] \to H^{-1}(\mathbb R^n)$ is defined by
$$F_2(t)=t^{\beta-1} f_2(t) \varphi(x).$$
Finally, the term $G\colon [0,T] \to L_2(U;H^{-1}(\mathbb R^n))$ is defined by
$G(t)=g(t).$

As  Example 1, it is easily seen that   
$$ G\in B([0,T];L_2(U;H^{-1}(\mathbb R^n)) \quad \text{ and } \quad F_2 \in \mathcal F^{\beta,\sigma} ((0,T];H^{-1}(\mathbb R^n)). $$ 
Therefore, all the structural  assumptions are satisfied in $H^{-1}(\mathbb R^n)$. By using Theorems  \ref{Thm5}, \ref{Thm7} and \ref{Thm6},  similar claims as in Example 1 are obtained.

\end{document}